\documentclass[11pt]{scrartcl}
\usepackage{amsfonts,amsmath,amsthm,amssymb}
\usepackage{mathrsfs,mathtools}
\usepackage{enumerate}
\usepackage{enumitem}
\usepackage{hyperref}
\usepackage{esint}
\usepackage{graphicx}
\usepackage{bm}
\usepackage{commath}
\usepackage{esint}
\usepackage{color}

\usepackage{algorithm}
\usepackage{algpseudocode}

\newcommand{\abssec}[1]{\noindent\normalsize {\bfseries #1\quad }\ignorespaces}
\renewenvironment{abstract}{\abssec{Abstract}}{\par\vspace{.1in}}
\newenvironment{keywords}{\abssec{Key Words}}{\par\vspace{.1in}}
\newenvironment{AMS}{\abssec{AMS subject
		classification}}{\par\vspace{.1in}}

\DeclareMathAlphabet{\mathpzc}{OT1}{pzc}{m}{it}

\usepackage[textsize=small]{todonotes}
\numberwithin{equation}{section}

\theoremstyle{plain}
\newtheorem{theorem}{Theorem}[section]

\newtheorem{lemma}[theorem]{Lemma}
\newtheorem{proposition}[theorem]{Proposition}

\theoremstyle{definition}
\newtheorem{remark}[theorem]{Remark}

\newtheorem{definition}[theorem]{Definition}

\title{Fractional Operators with Inhomogeneous Boundary Conditions:\\ Analysis, Control, and Discretization 
\thanks{The work of the first and second author is partially supported by  NSF grant DMS-1521590. }}

\author{Harbir Antil\thanks{Department of Mathematical Sciences, George Mason University, Fairfax, VA 22030, USA. \texttt{hantil@gmu.edu}}
\and
Johannes Pfefferer\thanks{Chair of Optimal Control, Technical University of Munich, Boltzmannstra\ss e 3, 85748 Garching by Munich, Germany. \texttt{pfefferer@ma.tum.de}}
\and
Sergejs Rogovs\thanks{Institut f\"ur Mathematik und Bauinformatik, Universit\"at der Bundeswehr M\"unchen, 85577 Neubiberg, Germany. \texttt{sergejs.rogovs@unibw.de}}
}

\begin{document}

\maketitle

\begin{abstract}
In this paper we introduce new characterizations of the spectral fractional Laplacian to incorporate nonhomogeneous Dirichlet and Neumann boundary conditions. The classical cases with homogeneous boundary conditions arise as a special case. We apply our definition to fractional elliptic equations of order $s \in (0,1)$ with nonzero Dirichlet and Neumann boundary condition. Here the domain $\Omega$ is assumed to be a bounded, quasi-convex Lipschitz domain.
To impose the nonzero boundary conditions, we construct fractional harmonic extensions of the boundary data. It is shown that solving for the fractional harmonic extension is equivalent to solving for the standard harmonic extension in the very-weak form. The latter result is of independent interest as well. 
The remaining fractional elliptic problem (with homogeneous boundary data) can be realized using the existing techniques.
We introduce finite element discretizations and derive discretization error estimates in natural norms, which are confirmed by numerical experiments. 
We also apply our characterizations to Dirichlet and Neumann  boundary optimal control problems with fractional elliptic equation as constraints. 
\end{abstract}



%%%%%%%%%%%%%%%%%%%%%%%%%%%%%%%%%%%%%%%%%%%%%%%%%%%%%%%%%%%%%%%%%%%%%%%%%%%%%%%%%%%%%%%%%%%%%%%%%%%%%%%%%%%%%%%%%%%%%%%%%%
\section{Introduction}\label{s:intro}
%%%%%%%%%%%%%%%%%%%%%%%%%%%%%%%%%%%%%%%%%%%%%%%%%%%%%%%%%%%%%%%%%%%%%%%%%%%%%%%%%%%%%%%%%%%%%%%%%%%%%%%%%%%%%%%%%%%%%%%%%%
Let $\Omega \subset \mathbb{R}^n$ with $n\ge 2$ be a bounded open set with boundary $\partial\Omega$. We will specify the regularity of the boundary in the sequel. The 
purpose of this paper is to study existence, 
uniqueness, regularity, and finite element approximation of the following 
nonhomogeneous Dirichlet boundary value problem
\begin{align}\label{eq:fracLap}
\begin{aligned}
 (-\Delta_D)^s u &= f \quad \mbox{in } \Omega , \\
  u &= g \quad \mbox{on } \partial\Omega . 
\end{aligned}
\end{align}
Here $g$ and $f$ are measurable functions on $\partial\Omega$ and $\Omega$ respectively, and satisfy certain conditions (that we shall specify later), $s \in (0,1)$ and $(-\Delta_D)^s$ denotes the modified spectral fractional Laplace operator with \emph{nonzero} boundary conditions. 

The nonlocality of $(-\Delta_D)^s$ makes \eqref{eq:fracLap} challenging. Nevertheless,  when $g \equiv 0$ the definition of the resulting nonlocal operator $(-\Delta_{D,0})^s$ incorporates the zero boundary conditions and has been well studied, see \cite{LCaffarelli_LSilvestre_2007a, PRStinga_JLTorrea_2010a, CBrandle_EColorado_APablo_USSanchez_2013a,XCabre_JTan_2010a, ACapella_JDavila_LDupaigne_YSire_2011a, MR3489634,RHNochetto_EOtarola_AJSalgado_2014a}, however the case $g\neq 0$ has been neglected by all these references.
Imposing nonzero boundary conditions in the nonlocal setting is highly nontrivial, 
which is the purpose of our paper. We will accomplish this by introducing a new characterization of $(-\Delta_D)^s$. 
We define our operator as
\begin{equation}\label{eq:intro_def}
 (-\Delta_D)^s u := \sum_{k=1}^\infty \left(\lambda_k^s \int_\Omega u \varphi_k 
     + \lambda_k^{s-1} \int_{\partial\Omega} u \partial_\nu \varphi_k \right) 
     \varphi_k,
\end{equation}
where $\lambda_k$ and $\varphi_k$ denote the eigenvalues and eigenfunctions of the Dirichlet Laplacian, see Section \ref{s:nzbc} for details. Obviously, if $u=0$ almost everywhere on the boundary then $(-\Delta_D)^s=(-\Delta_{D,0})^s$, see also Proposition~\ref{p:presvD}. Notice at this point that one cannot apply $(-\Delta_{D,0})^s$ to functions 
with nonzero boundary conditions as long as measuring traces is reasonable, see Section~\ref{s:cont_ex} for details. In contrast, in Section~\ref{s:cont_ex_1}, we will 
exemplarily illustrate that $(-\Delta_{D})^s 1$ is meaningful in that case. Moreover, if we set $s=1$ then \eqref{eq:intro_def} gives us 
\[
-\Delta u := \sum_{k=1}^\infty \left(\lambda_k \int_\Omega u \varphi_k 
+  \int_{\partial\Omega} u \partial_\nu \varphi_k \right) 
\varphi_k,
\]
i.e., the spectral characterization of the standard Laplacian (see Proposition~\ref{eq:Lap}
for details). In other words, our characterization of $(-\Delta_D)^s$ is a natural extension
of $-\Delta$. In addition, in Proposition~\ref{p:presvD}, we will 
see that the semigroup property $(-\Delta_D)^s (-\Delta_D)^{1-s} = -\Delta$ is valid.

To the best of our knowledge, there is only one work which is concerned with inhomogeneous boundary conditions in the context of the spectral fractional Laplacian. More precisely, in \cite{NALD} the authors study well-posedness of 
\begin{equation}\label{eq:fracLap_AB}
 \begin{aligned}
(-\Delta_{D,0})^s u &= f \quad \mbox{in } \Omega, \\
  u/\xi &= g \quad \mbox{on } \partial\Omega , 
  \end{aligned}
\end{equation}
where $\xi$ is a reference function with a prescribed singular behavior at the boundary. The (very) weak formulation of \eqref{eq:fracLap_AB} is given by
\[
	\int_{\Omega} u(-\Delta_{D,0})^sv=\int_{\Omega}fv-\int_{\partial\Omega}g\partial_{\nu}v\quad \forall v\in (-\Delta_{D,0})^{-s}C^\infty_0(\Omega),
\]
see \cite[Definition 3]{NALD}. This formulation even allows for data in measure spaces.
However, we emphasize that they do not impose a boundary condition of the type $u=g$ but 
instead consider $u/\xi = g$, where $\xi$ is a reference function with a prescribed singular behavior at the boundary. 

In contrast, in Section \ref{s:Ntrthm}, we will prove the integration-by-parts type formula
\begin{align*}
\int_{\Omega}(-\Delta_D)^suv&=
\int_\Omega u (-\Delta_{D,0})^sv+\int_{\partial\Omega}u \partial_\nu w_v,
\end{align*}	
where $w_v$ is defined as the solution to
\begin{align*}
(-\Delta_{D,0})^{1-s}w_v&=v\quad \text{in } \Omega,\\
 w_v&=0\quad\text{on }\partial \Omega.
\end{align*}
Based on this, we can show that the (very) weak formulation of \eqref{eq:fracLap} is given by
\[
\int_{\Omega} u(-\Delta_{D,0})^sv=\int_{\Omega}fv-\int_{\partial\Omega}g\partial_{\nu}w_v\quad \forall v\in \mathbb{H}^{2s}(\Omega).
\]
Thus, the condition $u=g$ within our formulation can be interpreted as a Dirichlet boundary condition.

From a practical point of view but also for the purpose of analyzing problem \eqref{eq:fracLap}, at first, we use a standard lifting argument by constructing a \emph{fractional harmonic map}
\begin{equation}\label{eq:vintro}
  (-\Delta_D)^s v = 0 \quad \mbox{in } \Omega , \quad v = g \quad \mbox{on } \partial\Omega . 
\end{equation}
It may seem at first glance that solving \eqref{eq:vintro} is as complicated as solving the original problem \eqref{eq:fracLap}. However, we show that  
solving \eqref{eq:vintro} is equivalent to solving 
\begin{equation}\label{eq:vintro_vw}
 \int_\Omega v (-\Delta \varphi) = -\int_{\partial\Omega} g\partial_\nu\varphi \quad \forall \varphi \in \operatorname{dom}(-\Delta), 
\end{equation}
i.e., the standard Laplace equation in the \emph{very-weak form}. 
To get $u$, it remains to find $w$ solving 
\begin{equation}\label{eq:wintro}
 (-\Delta_{D,0})^s w = f \quad \mbox{in } \Omega , \quad w = 0 \quad \mbox{on } \partial\Omega ,
\end{equation}
then $u = w+v$. Thus instead of looking for $u$  directly, we are reduced to solving  \eqref{eq:vintro_vw} and \eqref{eq:wintro} for $v$ and $w$, respectively.

Both \eqref{eq:vintro_vw} and \eqref{eq:wintro} have received a great deal of 
attention, we only refer to \cite{LionsMagenes1968, 
	MR2084239,CasasRaymond2006,MR3070527,MR3432846,NUM:NUM22057,apelnicaisepfefferer:2015-2extended}
	 for the first case and \cite{LCaffarelli_LSilvestre_2007a, 
	PRStinga_JLTorrea_2010a, 
	CBrandle_EColorado_APablo_USSanchez_2013a,XCabre_JTan_2010a, 
	ACapella_JDavila_LDupaigne_YSire_2011a, 
	MR3489634,RHNochetto_EOtarola_AJSalgado_2014a,
	ABonito_JEPasciak_2015a,  
	HAntil_JPfefferer_MWarma_2016a, meidner2017hp} for the latter. We will show 
	that both \eqref{eq:vintro_vw} and \eqref{eq:wintro} are well-posed 
	(solution exists and is unique) thus \eqref{eq:fracLap} is well-posed as 
	well. 

For the numerical computation of solutions of \eqref{eq:vintro_vw}, we rely 
on well established techniques, see for instance
\cite{MR2084239,NUM:NUM22057,apelnicaisepfefferer:2015-2extended}.
It is even possible to apply a standard finite element method especially if the boundary 
datum $g$ is regular enough. However, the numerical realization of the nonlocal operator $(-\Delta_{D,0})^s$ in \eqref{eq:wintro} is more challenging. Several approaches have been advocated, for instance, computing the eigenvalues and eigenvectors of $-\Delta_{D,0}$ (cf.~\cite{MR3679919}), Dunford-Taylor integral representation \cite{ABonito_JEPasciak_2015a}, or numerical schemes based on the Caffarelli-Silvestre (or the Stinga-Torrea) extension, just to name a few. In our work, we choose the latter even though the proposed ideas directly apply to other approaches where $(-\Delta_{D,0})^s$ appears, for instance \cite{ABonito_JEPasciak_2015a}. Notice that the aforementioned extension of Caffarelli-Silvestre (or the Stinga-Torrea) is only applicable to $(-\Delta_{D,0})^s$ and not directly to the operator $(-\Delta_D)^s$ in \eqref{eq:fracLap}.

The extension approach was introduced in 
\cite{LCaffarelli_LSilvestre_2007a} for $\mathbb{R}^n$, see its extensions to 
bounded domains \cite{ACapella_JDavila_LDupaigne_YSire_2011a, 
	PRStinga_JLTorrea_2010a}. It states that $(-\Delta_{D,0})^s$ can be realized as 
an operator that maps a Dirichlet boundary condition to a Neumann condition via 
an extension problem on the semi-infinite cylinder $\mathcal{C} = \Omega \times 
(0,\infty)$, i.e., a Dirichlet-to-Neumann operator. A first finite element method to solve \eqref{eq:wintro} based on the extension approach is given in \cite{RHNochetto_EOtarola_AJSalgado_2014a}. This was applied to semilinear problems in \cite{HAntil_JPfefferer_MWarma_2016a}.
In the context of fractional distributed optimal control problems, the extension approach was considered in \cite{HAntil_EOtarola_2014a} where related discretization error estimates are established as well.

An additional advantage is that our characterization allows for imposing other types of nonhomogeneous boundary conditions such as Neumann boundary conditions (see sections \ref{s:nn} and \ref{s:app_n}) and that it immediately extends to general second order fractional operators (see Section~\ref{s:fext}).

We remark that the difficulties in imposing the nonhomogeneous boundary conditions are not limited to the spectral fractional Laplacian. In fact, the integral definition of fractional Laplacian \cite{GA:JP15} requires imposing boundary conditions on $\mathbb{R}^n \setminus \Omega$. On the other hand, 
the so-called regional definition of fractional Laplacian with nonhomogeneous boundary conditions may lead to an ill-posed problem when $s\le 1/2$, see \cite{MR2167307, MR3306694}.

This paper is organized as follows: We state the definitions of $(-\Delta_{D,0})^s$ and $(-\Delta_{N,0})^s$ in Sections~\ref{s:zbc} and \ref{s:zbc_n}. Moreover, we introduce the relevant function spaces. Some of the material in these sections is well-known. However, we recall it such that the paper is self-contained. Our main work begins from Section~\ref{s:nzbc} where we first state the new characterization of Dirichlet fractional Laplacian. Next we discuss the Neumann case in Section~\ref{s:nn}. In Section~\ref{s:Ntrthm} we state two not so well known trace theorems for $H^2(\Omega)$ functions in bounded Lipschitz domains and prove integration-by-parts type formulas for the spectral fractional Laplacians. Subsequently, in Section~\ref{s:app}, we analyze the boundary value problem \eqref{eq:fracLap} and derive a priori finite element error estimates. In Section~\ref{s:app_n}, we study corresponding results for the nonhomogeneous Neumann problem. Afterwards, in Section~\ref{s:nbc}, we consider Dirichlet and Neumann boundary optimal control problems with fractional elliptic PDEs as constraints. We verify our theoretical rates of convergence via two numerical examples in Section~\ref{s:numerics}. We provide further extensions to general second order elliptic operators in Section~\ref{s:fext}.

%%%%%%%%%%%%%%%%%%%%%%%%%%%%%%%%%%%%%%%%%%%%%%%%%%%%%%%%%%%%%%%%%%%%%%%%%%%%%%%%%%%%%%%%%%%%%%%%%%%%%%%%%%%%%%%%%%%%%%%%%%
\section{Spectral Fractional Laplacian}
%%%%%%%%%%%%%%%%%%%%%%%%%%%%%%%%%%%%%%%%%%%%%%%%%%%%%%%%%%%%%%%%%%%%%%%%%%%%%%%%%%%%%%%%%%%%%%%%%%%%%%%%%%%%%%%%%%%%%%%%%%

In this section, without any specific mention, we will assume that the boundary $\partial\Omega$ is Lipschitz continuous. 

%%%%%%%%%%%%%%%%%%%%%%%%%%%%%%%%%%%%%%%%%%%%%%%%%%%%%%%%%%%%%%%%%%%%%%%%%%%%%%%%%%%%%%%%%%%%%%%%%%%%%%%%%%%%%%%%%%%%%%%%%%
\subsection{Zero Dirichlet Boundary Data}\label{s:zbc}
%%%%%%%%%%%%%%%%%%%%%%%%%%%%%%%%%%%%%%%%%%%%%%%%%%%%%%%%%%%%%%%%%%%%%%%%%%%%%%%%%%%%%%%%%%%%%%%%%%%%%%%%%%%%%%%%%%%%%%%%%%

Let $-\Delta_{D,0}$ be the realization in $L^2(\Omega)$ of the Laplace operator 
with zero  Dirichlet boundary condition. It is well-known that 
$-\Delta_{D,0}$ has a compact resolvent and its eigenvalues form a 
non-decreasing sequence $0<\lambda_1\le\lambda_2\le\cdots\le\lambda_k\le\cdots$ 
satisfying $\lim_{k\to\infty}\lambda_k=\infty$. We denote by $\varphi_k \in 
H^1_0(\Omega)$ the orthonormal eigenfunctions associated with $\lambda_k$. It is well known that these eigenfunctions form an orthonormal basis of $L^2(\Omega)$.

For $0<s<1$,  we define the fractional order Sobolev space
\begin{align*}
H^s(\Omega):=\left\{u\in L^2(\Omega):\; 
\int_{\Omega}\int_{\Omega}\frac{|u(x)-u(y)|^2}{|x-y|^{n+2s}}\;dxdy<\infty\right\},
\end{align*}
and we endow it with the norm defined by
\begin{align}\label{eq:HsN}
\|u\|_{H^s(\Omega)}=\left(\|u\|_{L^2(\Omega)}^2+|u|_{H^s(\Omega)}^2\right)^{\frac
 12},
\end{align}
where the semi-norm $|u|_{H^s(\Omega)}$ is defined by
\begin{equation}\label{norm-sob-es}
|u|_{H^s(\Omega)}^2=\int_{\Omega}\int_{\Omega}\frac{|u(x)-u(y)|^2}{|x-y|^{n+2s}}\;dxdy.
\end{equation}
The fractional order Sobolev spaces $H^{t}(\partial \Omega)$ on the boundary 
with $0<t<1$ are defined in 
the same manner.
We also let 
\begin{align*}
H_0^s(\Omega):=\overline{\mathcal D(\Omega)}^{H^s(\Omega)},
\end{align*}
where $\mathcal D(\Omega)$ denotes the space of test functions on $\Omega$, that is, the space of infinitely continuously differentiable functions with compact support in $\Omega$, and
\begin{align*}
H_{00}^{\frac 12}(\Omega):=\left\{u\in 
H^{\frac 12}(\Omega):\;\int_{\Omega}\frac{u^2(x)}{\mbox{dist}(x,\partial\Omega)}\;dx<\infty\right\} ,
\end{align*}
with norm
\[
\|u\|_{H_{00}^{\frac 
12}(\Omega)}=\left(\|u\|_{H^{\frac12}(\Omega)}^2+\int_{\Omega}\frac{u^2(x)}{\mbox{dist}(x,\partial\Omega)}\;dx\right)^{\frac
	12}.
\]
We further introduce the dual spaces of 
$H^s_0(\Omega)$ and 
$H^t(\partial \Omega)$, and denote them by $H^{-s}(\Omega)$ and 
$H^{-t}(\partial \Omega)$, respectively.

For any $s\geq0$, we also define the following fractional order Sobolev space
\begin{align*}
\mathbb H^s(\Omega):=\left\{u=\sum_{k=1}^\infty u_k\varphi_k\in L^2(\Omega):\;\;\|u\|_{\mathbb H^s(\Omega)}^2:=\sum_{k=1}^\infty \lambda_k^su_k^2<\infty\right\},
\end{align*}
where we recall that $\lambda_k$ are the eigenvalues of $-\Delta_{D,0}$ with associated normalized eigenfunctions $\varphi_k$ and
\begin{align*} 
u_k=(u,\varphi_k)_{L^2(\Omega)}=\int_{\Omega}u\varphi_k.
\end{align*}

It is well-known that
\begin{equation}\label{inf}
\mathbb H^s(\Omega)=
\begin{cases}
H^s(\Omega)=H_0^s(\Omega)\;\;\;&\mbox{ if }\; 0<s<\frac 12,\\
H_{00}^{\frac 12}(\Omega)\;\;&\mbox{ if }\; s=\frac 12,\\
H_0^s(\Omega)\;\;&\mbox{ if }\; \frac 12<s<1.
\end{cases}
\end{equation}
The dual space of $\mathbb{H}^s(\Omega)$ will be denoted by 
$\mathbb{H}^{-s}(\Omega)$.

The fractional order Sobolev spaces can be also defined by using interpolation theory. That is, for every $0<s<1$,
\begin{align}\label{eq:inf_interp_Hs}
H^s(\Omega)=[H^1(\Omega),L^2(\Omega)]_{1-s},
\end{align}
and
\begin{align}\label{inf_interp}
H_0^s(\Omega)=[H_0^1(\Omega),L^2(\Omega)]_{1-s}\;\mbox{ if }\; s\in 
(0,1)\setminus\{1/2\}\;\mbox{ and }\; H_{00}^{\frac 
12}=[H_0^1(\Omega),L^2(\Omega)]_{\frac 12}.
\end{align}

\begin{definition}\label{d:zbc}
The spectral fractional Laplacian is defined on the space $C^\infty_0(\Omega)$ by
\begin{align*}
(-\Delta_{D,0})^su:=\sum_{k=1}^\infty\lambda_k^su_k\varphi_k\quad 
\text{with}\quad u_k=\int_{\Omega}u\varphi_k.
\end{align*}
\end{definition}
By observing that
\[
	\int_\Omega (-\Delta_{D,0})^su v = \sum_{k=1}^\infty 
	\lambda_k^s u_k v_k =\sum_{k=1}^\infty 
	\lambda_k^{s/2} u_k \lambda_k^{s/2} v_k\leq 
	\|u\|_{\mathbb{H}^s(\Omega)}\|v\|_{\mathbb{H}^s(\Omega)}
\]
for any $v=\sum_{k=1}^\infty v_k \varphi_k\in \mathbb{H}^s(\Omega)$,
the operator $(-\Delta_{D,0})^s$ extends to an operator mapping from $\mathbb 
H^s(\Omega)$ to $\mathbb H^{-s}(\Omega)$ by density. Moreover, we notice that 
in this case we have
\begin{equation}
 \| u \|_{\mathbb{H}^s(\Omega)} = \| (-\Delta_{D,0})^{\frac{s}{2}} u \|_{L^2(\Omega)}. 
\end{equation}

In addition, the following estimate holds by definition of the spaces 
$H^s(\Omega)$, $H_0^s(\Omega)$, $H_{00}^{\frac12}(\Omega)$, and 
$\mathbb{H}^s(\Omega)$, and by relation \eqref{inf}.

\begin{proposition}\label{l:HsBHs}
If $u \in \mathbb{H}^s(\Omega)$ with $0<s<1$ then there exists a constant $C = C(\Omega,s) > 0$ such that
\[
 \|u\|_{H^s(\Omega)} \le C \|u\|_{\mathbb{H}^s(\Omega)} .  
\]
\end{proposition}

%%%%%%%%%%%%%%%%%%%%%%%%%%%%%%%%%%%%%%%%%%%%%%%%%%%%%%%%%%%%%%%%%%%%%%%%%%%%%%%%%%%%%%%%%%%%%%%%%%%%%%%%%%%%%%%%%%%%%%%%%%
\subsubsection{A Counter Example}\label{s:cont_ex}
%%%%%%%%%%%%%%%%%%%%%%%%%%%%%%%%%%%%%%%%%%%%%%%%%%%%%%%%%%%%%%%%%%%%%%%%%%%%%%%%%%%%%%%%%%%%%%%%%%%%%%%%%%%%%%%%%%%%%%%%%%
The purpose of this example is to illustrate that in general 
Definition~\ref{d:zbc_n} cannot be applied to functions with nonzero 
boundary conditions as long as measuring traces is reasonable. Towards this end, we first observe that the operator $(-\Delta_{D,0})^s$ can be extended to an operator mapping from $\mathbb{H}^t(\Omega)$ to $\mathbb{H}^{t-2s}(\Omega)$, see the foregoing explanations for the special case $t=s$. Next, we apply this operator to the function $u\equiv 1$ as follows: We set $\Omega=(0,1)$. Thus, there holds $\varphi_k  = \sin(k\pi x)$ and $\lambda_k = k^2\pi^2$. Basic calculations yield
\[
 u_k = \left\{\begin{array}{ll}
                0 & \mbox{if } k\mbox{ even}, \\
                \frac{2}{k\pi} & \mbox{if } k\mbox{ odd} .
              \end{array}
       \right.
\]
Moreover, we get for $t\geq 1/2$
\begin{align*}
	\|(-\Delta_{D,0})^s1\|_{\mathbb{H}^{t-2s}(\Omega)}^2&=\sum_{k=1}^\infty\lambda_k^{t-2s}\left(\int_\Omega (-\Delta_{D,0})^s1 \varphi_k\right)^2=4\pi^{2t-2}\sum_{k=1}^\infty(2k-1)^{2t-2}\\
	&=4\pi^{2t-2}\sum_{k=1}^\infty(2k-1)^{2t-1}(2k-1)^{-1}\geq 4\pi^{2t-2}\sum_{k=1}^\infty(2k-1)^{-1}\\
	&\geq 2\pi^{2t-2}\sum_{k=1}^\infty k^{-1}.
\end{align*}
We observe that the series on the right hand side of the above inequality is not convergent. Consequently, whenever it is possible to measure the Dirichlet trace in the classical sense, the fractional operator from Definition \ref{d:zbc_n} cannot be applied. In \cite[Introduction]{NALD} it is shown that in case of $t=0$ the application of the fractional Laplacian to the function $1$ yields the killing measure.
In Section~\ref{s:cont_ex_1} we will show that with our definition of 
fractional Laplacian introduced in Section~\ref{s:nzbc}, the issues discussed
above can be fixed.

%%%%%%%%%%%%%%%%%%%%%%%%%%%%%%%%%%%%%%%%%%%%%%%%%%%%%%%%%%%%%%%%%%%%%%%%%%%%%%%%%%%%%%%%%%%%%%%%%%%%%%%%%%%%%%%%%%%%%%%%%%
\subsection{Zero Neumann Boundary Data}\label{s:zbc_n}
%%%%%%%%%%%%%%%%%%%%%%%%%%%%%%%%%%%%%%%%%%%%%%%%%%%%%%%%%%%%%%%%%%%%%%%%%%%%%%%%%%%%%%%%%%%%%%%%%%%%%%%%%%%%%%%%%%%%%%%%%%

Let $-\Delta_{N,0}$ be the realization of the Laplace operator with zero 
Neumann boundary condition. 
It is well-known that there exists a sequence of nonnegative eigenvalues 
$\{\mu_k\}_{k\ge 1}$ satisfying $0=\mu_1 < \mu_2 \le \cdots \le\mu_k \le \cdots $ 
with $\lim_{k\rightarrow \infty} \mu_k = \infty$ and corresponding eigenfunctions
$\{\psi_k\}_{k \ge 1}$ in $H^1(\Omega)$. 
We have that $\mu_1 = 0$, $\psi_1 = 1/\sqrt{|\Omega|}$, 
$\int_\Omega \psi_k = 0$ for all $k \ge 2$. Moreover, the eigenfunctions 
$\{\psi_k\}_{k\ge 1}$ form an orthonormal basis of $L^2(\Omega)$. 

For any $s\geq 0$ we define the fractional order Sobolev spaces $H^s_{\int}(\Omega)$ \cite{MR3489634}:
\begin{align*}
H^s_{\int}(\Omega):=\left\{u=\sum_{k=2}^\infty u_k\psi_k\in L^2(\Omega):\; 
\|u\|_{H^s_{\int}(\Omega)}^2:=\sum_{k=2}^\infty \mu_k^s u_k^2<\infty\right\}.
\end{align*} 
Notice that any function $u$ belonging to $H^s_{\int}(\Omega)$ fulfills 
$\int_\Omega u 
=0$.

Furthermore, we denote by $H^{-s}_{\int}(\Omega)$ the dual space of 
$H^s_{\int}(\Omega)$.

\begin{definition}\label{d:zbc_n}
For any function $u\in C^\infty(\bar\Omega)$ with $\partial_\nu u =0$ we define the spectral fractional Laplacian by
\begin{align*}
(-\Delta_{N,0})^su:=\sum_{k=2}^\infty\mu_k^su_k\psi_k \quad \text{with} \quad  u_k=\int_{\Omega}u\psi_k.
\end{align*}
\end{definition}
As in the previous section for the Dirichlet Laplacian, the operator 
$(-\Delta_{N,0})^s$ can be extended to an 
operator mapping from $H^s_{\int}(\Omega)$ to $H^{-s}_{\int}(\Omega)$. Notice 
as 
well that $\int_\Omega (-\Delta_{N,0})^su = 0$ and 
$
 \| (-\Delta_{N,0})^{\frac{s}{2}} u \|_{L^2(\Omega)}^2 = \sum_{k=2}^\infty 
 \mu^s_k u_k^2 .
$ 
Thus we have 
\begin{equation}\label{eq:Hint2}
 \|u\|_{H^s_{\int}(\Omega)} =  \| (-\Delta_{N,0})^{\frac{s}{2}} u \|_{L^2(\Omega)}.
\end{equation}
We next recall \cite[Lemma~7.1]{MR3489634}.

\begin{proposition}\label{prop:normequi}
For any $0 < s < 1$, $u \in H^s_{\int}(\Omega)$ if and only if $u \in 
H^s(\Omega)$ and $\int_\Omega u =0$. In addition, the norm in \eqref{eq:Hint2} 
is equivalent to the seminorm $|\cdot|_{H^s(\Omega)}$ defined in 
\eqref{norm-sob-es}. 
\end{proposition}
\begin{remark}
	According to Proposition \ref{prop:normequi}, we have that
	\[
		\|u\|_{H^s(\Omega)}^2\sim u_1^2+\sum_{k=2}^\infty (1+\mu_k^s) u_k^2.
	\]
	Due to this, the Neumann Laplacian from Definition \ref{d:zbc_n} is also extendable to an operator mapping from $H^s(\Omega)$ to $H^s(\Omega)^*$. However, since we are going to treat associated boundary value problems, we consider the Neumann Laplacian with the mapping properties from above to ensure uniqueness of the solution.
\end{remark}

%%%%%%%%%%%%%%%%%%%%%%%%%%%%%%%%%%%%%%%%%%%%%%%%%%%%%%%%%%%%%%%%%%%%%%%%%%%%%%%%%%%%%%%%%%%%%%%%%%%%%%%%%%%%%%%%%%%%%%%%%%
\subsection{Nonzero Dirichlet Boundary Data}\label{s:nzbc}
%%%%%%%%%%%%%%%%%%%%%%%%%%%%%%%%%%%%%%%%%%%%%%%%%%%%%%%%%%%%%%%%%%%%%%%%%%%%%%%%%%%%%%%%%%%%%%%%%%%%%%%%%%%%%%%%%%%%%%%%%%

To motivate our definition of fractional Laplacian with nonzero Dirichlet boundary datum, we first derive  such a characterization for the standard Laplacian. 
\begin{proposition}
For any $u \in C^\infty(\bar\Omega)$ we have that
\begin{equation}\label{eq:Lap}
 -\Delta_D u := \sum_{k=1}^\infty \left(\lambda_k \int_\Omega u \varphi_k 
     + \int_{\partial\Omega} u \partial_\nu \varphi_k \right) \varphi_k
\end{equation}
fulfills
$-\Delta_D u = -\Delta u$ a.e. in $\Omega$.
\end{proposition}
\begin{proof} 
Standard integration-by-parts formula yields
\begin{align}
\int_{\Omega} -\Delta u \varphi_k &= \int_{\Omega} \nabla u \cdot \nabla 
\varphi_k
- \int_{\partial \Omega} \partial_\nu u \underbrace{\varphi_k}_{=0} \notag\\
&= \int_{\Omega} -\Delta \varphi_k u + \int_{\partial \Omega} u 
\partial_\nu \varphi_k\notag\\
&= \lambda_k\int_{\Omega} \varphi_k u + \int_{\partial \Omega} u \partial_\nu 
\varphi_k,\label{eq:dericlassical}
\end{align}
where in the last equality we have used the fact that $\varphi_k$ are the eigenfunctions of the Laplacian with eigenvalues $\lambda_k$. This yields the desired characterization having in mind that the eigenfunctions form an orthonormal basis of $L^2(\Omega)$.
\end{proof}
By density results, the operator $-\Delta_D$ extends to an operator 
mapping from 
$H^1(\Omega)$ to $H^{-1}(\Omega)$ in the classical way. 

We are now ready to state our definition of the fractional Laplacian 
$(-\Delta_D)^s$. 

\begin{definition}[nonzero Dirichlet]\label{d:nzbc}
We define the spectral fractional Laplacian on $C^\infty(\bar\Omega)$ by 
\begin{equation}\label{eq:fLap_g}
(-\Delta_D)^s u := \sum_{k=1}^\infty \left(\lambda_k^s \int_\Omega u \varphi_k 
     + \lambda_k^{s-1} \int_{\partial\Omega} u \partial_\nu \varphi_k \right) 
     \varphi_k.
\end{equation} 
\end{definition}
Let us set
\[
	u_{\Omega,k}=\int_\Omega u \varphi_k\quad\text{and}\quad 
	u_{\partial\Omega,k}=\int_{\partial\Omega} u \partial_\nu \varphi_k.
\]
We observe that for any $v=\sum_{k=1}^{\infty}v_k\varphi_k\in 
\mathbb{H}^{s}(\Omega)$ there holds
\begin{align*}
\int_\Omega (-\Delta_D)^s u v 
&=\sum_{k=1}^\infty 
\left(\lambda_k^s 
u_{\Omega,k} 
+ \lambda_k^{s-1} u_{\partial\Omega,k} 
\right)v_k=\sum_{k=1}^\infty\lambda_k^{\frac{s}2} 
\left(u_{\Omega,k} + \lambda_k^{-1} u_{\partial\Omega,k} 
\right)\lambda_k^{\frac{s}2}v_k\\
&\leq \left(\sum_{k=1}^\infty 
\lambda_k^{s}\left(u_{\Omega,k} + \lambda_k^{-1} u_{\partial\Omega,k} 
\right)^2\right)^{1/2}\|v\|_{\mathbb{H}^{s}(\Omega)},
\end{align*}
where we used the orthogonality of the eigenfunctions $\varphi_k$. Thus the 
operator 
$(-\Delta_D)^s$ can be extended to an operator mapping from
\[
\mathbb{D}^{s}(\Omega):=\Big\{u\in L^2(\Omega):\; 
\sum_{k=1}^\infty 
\lambda_k^{s}\left(u_{\Omega,k} + \lambda_k^{-1} u_{\partial\Omega,k} 
\right)^2<\infty\Big\}
\]
to $\mathbb{H}^{-s}(\Omega)$, cf. Section \ref{s:app} where we solve 
associated boundary value problems.

We notice that Definition~\ref{d:nzbc} obeys the 
following fundamental properties. 

\begin{proposition}\label{p:presvD}
Let $(-\Delta_D)^s$ be as in Definition \ref{d:nzbc} then the following holds:
\begin{enumerate}
\item When $s = 1$ we obtain the standard Laplacian \eqref{eq:Lap}.
\item For any $u \in C^\infty_0(\Omega)$ 
there holds \[(-\Delta_D)^s u = 
(-\Delta_{D,0})^s u\] a.e. in $\Omega$, i.e., we recover the Definition~\ref{d:zbc}.
\item For any $s\in(0,1)$ and any $u\in C^\infty(\bar\Omega)$ 
there is the identity \[(-\Delta_D)^s 
(-\Delta_D)^{1-s} u = - \Delta u\] a.e. in $\Omega$.
\end{enumerate}
\end{proposition}
\begin{proof}
The first two assertions are easy to check, thus we only elaborate on the last 
one. Let $u\in C^\infty(\bar\Omega)$ and define
\[
	v_l:=\sum_{k=1}^l\left(\lambda_k^{1-s} \int_\Omega u \varphi_k 
	+ \lambda_k^{-s} \int_{\partial\Omega} u \partial_\nu \varphi_k \right) 
	\varphi_k.
\]
Using the orthogonality of the eigenfunction $\varphi_k$ and \eqref{eq:dericlassical}, we obtain for any $t\in[0,\frac12)$ and for any $l,m\in \mathbb{N}$ with $l\geq m$
\begin{align*}
	\|v_l-v_m\|^2_{\mathbb{H}^{2s+t}(\Omega)}&=\sum_{k=m+1}^l\lambda_k^{2s+t}\left(\lambda_k^{1-s} \int_\Omega u \varphi_k 
	+ \lambda_k^{-s} \int_{\partial\Omega} u \partial_\nu \varphi_k \right)^2\\
	&=\sum_{k=m+1}^l\lambda_k^{t}\left(\lambda_k \int_\Omega u \varphi_k 
	+ \int_{\partial\Omega} u \partial_\nu \varphi_k \right)^2\\
	&=\sum_{k=m+1}^l\lambda_k^{t}\left(\int_\Omega -\Delta u \varphi_k  \right)^2\leq \|\Delta u\|^2_{\mathbb{H}^{t}(\Omega)},
\end{align*}
where the last term is bounded independent of $l$ and $m$ since $\Delta u\in H^{t}(\Omega)=\mathbb{H}^{t}(\Omega)$. Thus, according to the Cauchy criterion, the limit $(-\Delta_D)^{1-s}u:=\lim_{l\rightarrow\infty}v_l$ exists in $\mathbb{H}^{2s+t}(\Omega)$. Moreover, we can choose the parameter $t$ such that $2s+t>\frac12$ which implies zero trace of $(-\Delta_D)^{1-s}u$ according to the definition of $\mathbb{H}^s(\Omega)$. Combining the last two observations, we are able to apply $(-\Delta_D)^s$ to $(-\Delta_D)^{1-s}u$. To this end, we define
\[
	w_l=\sum_{k=1}^l \lambda_k^{s}\left( \int_\Omega (-\Delta_D)^{1-s}u \varphi_k \right) 
	\varphi_k.
\]
As before we deduce for any $l,m\in\mathbb{N}$ with $l\geq m$
\begin{align*}
	\|w_l-w_m\|_{L^2(\Omega)}^2&=\sum_{k=m+1}^l\lambda_k^{2s}\left( \int_\Omega (-\Delta_D)^{1-s}u \varphi_k \right)^2\leq \|(-\Delta_D)^{1-s}u\|^2_{\mathbb{H}^{2s}(\Omega)},
\end{align*}
where the last term is bounded independent of $l$ and $m$ since $(-\Delta_D)^{1-s}u\in \mathbb{H}^{2s+t}(\Omega)$. Again, due to the Cauchy criterion, the limit $(-\Delta_D)^{s}(-\Delta_D)^{1-s}u:=\lim_{l\rightarrow\infty}w_l$ exists in $L^2(\Omega).$ 
This allows us to conclude
\begin{align*}
 \int_{\Omega}(-\Delta_D)^s (-\Delta_D)^{1-s} u \varphi_k&=\lambda_k^s\int_{\Omega}(-\Delta_D)^{1-s} u \varphi_k\\
 &=\lambda_k^s \left(\lambda_k^{1-s} \int_\Omega u \varphi_k 
 + \lambda_k^{-s} \int_{\partial\Omega} u \partial_\nu \varphi_k \right)\\
 &=\left(\lambda_k \int_\Omega u \varphi_k 
 +\int_{\partial\Omega} u \partial_\nu \varphi_k \right)\\
 &=\int_{\Omega} -\Delta u \varphi_k,
\end{align*}
where we used the orthogonality of the eigenfunctions $\varphi_k$ several times, and \eqref{eq:dericlassical}. Since $\{\varphi_k\}_{k\in \mathbb{N}}$ represents an orthonormal basis of $L^2(\Omega)$, we get the desired result. 
\end{proof}

%%%%%%%%%%%%%%%%%%%%%%%%%%%%%%%%%%%%%%%%%%%%%%%%%%%%%%%%%%%%%%%%%%%%%%%%%%%%%%%%%%%%%%%%%%%%%%%%%%%%%%%%%%%%%%%%%%%%%%%%%%
\subsubsection{A fix to the counter example}\label{s:cont_ex_1}
%%%%%%%%%%%%%%%%%%%%%%%%%%%%%%%%%%%%%%%%%%%%%%%%%%%%%%%%%%%%%%%%%%%%%%%%%%%%%%%%%%%%%%%%%%%%%%%%%%%%%%%%%%%%%%%%%%%%%%%%%%
Towards this end we shall apply our definition of fractional Laplacian
to the example discussed in Section~\ref{s:cont_ex}. Since $u\equiv 1$, 
we have
\begin{align*}
 (-\Delta_D)^s 1 
  &=\sum_{k=1}^\infty \left( \lambda_k^s u_k 
    + \lambda_k^{s-1}\int_{\partial\Omega} 1 \partial_\nu \varphi_k \right)\varphi_k .
\end{align*}
Recall that 
\[
 u_k = \left\{\begin{array}{ll}
                0 & \mbox{if } k\mbox{ even}, \\
                \frac{2}{k\pi} & \mbox{if } k\mbox{ odd} .
              \end{array}
       \right.
\]
It is easy to check that 
\[
  \partial_\nu \varphi_k = \left\{\begin{array}{ll} 
                                    k\pi (-1)^k & x=1, \\
                                    -k\pi       & x=0 .  
                                  \end{array} \right.
\]
This yields that
\[
 \int_{\partial\Omega} 1 \partial_\nu \varphi_k = k\pi (-1)^k -k\pi 
 = \left\{\begin{array}{ll} 
           0 & \mbox{if } k\mbox{ even}, \\
           -2k\pi  & \mbox{if } k\mbox{ odd} .  
          \end{array} \right.
\]
Then 
\begin{align*}
(-\Delta_D)^s 1 
 &= \sum_{k=1}^\infty \left( \lambda_{2k-1}^s \frac{2}{(2k-1)\pi} 
   + \lambda_{2k-1}^{s-1} (-2(2k-1)\pi)  \right)\varphi_k \\
 &= \sum_{k=1}^\infty \left( (2k-1)^{2s}\pi^{2s} \frac{2}{(2k-1)\pi} 
   + (2k-1)^{2(s-1)} \pi^{2(s-1)} (-2(2k-1)\pi)  \right)\varphi_k \\
 &=0 .   
\end{align*}

%%%%%%%%%%%%%%%%%%%%%%%%%%%%%%%%%%%%%%%%%%%%%%%%%%%%%%%%%%%%%%%%%%%%%%%%%%%%%%%%%%%%%%%%%%%%%%%%%%%%%%%%%%%%%%%%%%%%%%%%%%
\subsection{Nonzero Neumann Boundary Data}\label{s:nn}
%%%%%%%%%%%%%%%%%%%%%%%%%%%%%%%%%%%%%%%%%%%%%%%%%%%%%%%%%%%%%%%%%%%%%%%%%%%%%%%%%%%%%%%%%%%%%%%%%%%%%%%%%%%%%%%%%%%%%%%%%%
As in Section~\ref{s:nzbc}, in order to motivate our definition for the fractional 
nonhomogeneous Neumann Laplacian, we first derive such a characterization for the standard 
Laplacian.
\begin{proposition}
For any $u \in C^\infty(\bar\Omega)$ we have that
\begin{equation}\label{eq:Lap_n}
 \begin{split}
 -\Delta_N u &:= \sum_{k=2}^\infty \left(\mu_k \int_\Omega u \psi_k 
     - \int_{\partial\Omega} \partial_\nu u \psi_k \right) \psi_k+|\Omega|^{-1}\int_\Omega-\Delta u\\
     &= \sum_{k=2}^\infty \left(\mu_k \int_\Omega u \psi_k 
     - \int_{\partial\Omega} \partial_\nu u \psi_k \right) \psi_k - |\Omega|^{-1}\int_{\partial \Omega} \partial_\nu u.
 \end{split}
\end{equation}
fulfills $-\Delta_N u = -\Delta u$ a.e. in $\Omega$.
\end{proposition}
\begin{proof}
Standard integration-by-parts formula yields
\begin{align}
\int_{\Omega} -\Delta u \psi_k 
&= \int_{\Omega} \nabla u \cdot \nabla \psi_k
- \int_{\partial \Omega} \partial_\nu u \psi_k \notag\\
&= \int_{\Omega} -\Delta \psi_k u - \int_{\partial \Omega} \partial_\nu u  \psi_k + \int_{\partial \Omega} u \underbrace{\partial_\nu \psi_k}_{=0}\notag\\
&= \mu_k\int_{\Omega} \psi_k u - \int_{\partial \Omega} \partial_\nu u \psi_k,\label{eq:dericlassicN}
\end{align}
where in the last equality we have used the fact that $\psi_k$ are the eigenfunctions of Laplacian with eigenvalues $\mu_k$. This yields the desired representation of $-\Delta$ having in mind that $\{\psi_k\}_{k\in\mathbb{N}}$ represents an orthonormal basis of $L^2(\Omega)$, and that
\[\left(\int_\Omega-\Delta u \psi_1\right)\psi_1=|\Omega|^{-1}\int_\Omega-\Delta u=- |\Omega|^{-1}\int_{\partial\Omega}\partial_\nu u.\]
\end{proof}
As for the Dirichlet Laplacian, if $\int_{\partial\Omega}\partial_\nu u=0$, the operator $-\Delta_N$ can be extended to an operator mapping from 
$H_{\int}^1(\Omega)$ to $H_{\int}^{-1}(\Omega)$  in the classical way.
We are now ready to state our definition of fractional Laplacian with nonzero Neumann boundary conditions.

\begin{definition}[nonzero Neumann]\label{d:nzbc_n}
We define the spectral fractional Laplacian on $C^\infty(\bar{\Omega})$ by 
\begin{equation}\label{eq:fLap_g_n}
(-\Delta_N)^s u := \sum_{k=2}^\infty \left(\mu_k^s \int_\Omega u \psi_k 
     - \mu_k^{s-1} \int_{\partial\Omega} \partial_\nu u  \psi_k \right) \psi_k- |\Omega|^{-1}\int_{\partial\Omega}\partial_\nu u.
\end{equation}
\end{definition}
Note that we have $\int_{\Omega}(-\Delta_N)^s u=-\int_{\partial\Omega}\partial_\nu u$ by construction. However, different types of normalization are possible as well. Similar to the foregoing section, let us set
\[
u_{\Omega,k}=\int_\Omega u \psi_k\quad\text{and}\quad 
u_{\partial\Omega,k}=\int_{\partial\Omega} \partial_\nu u \psi_k.
\]
For any $v=\sum_{k=2}^\infty v_k\psi_k\in 
H_{\int}^s(\Omega)$ we observe that
\begin{align*}
\int_\Omega (-\Delta_{N})^s u v 
&=\sum_{k=2}^\infty 
\left(\mu_k^s 
u_{\Omega,k} 
- \mu_k^{s-1} u_{\partial\Omega,k} 
\right)v_k=\sum_{k=2}^\infty\mu_k^{s/2} 
\left(u_{\Omega,k} - \mu_k^{-1} u_{\partial\Omega,k} 
\right)\mu_k^{s/2}v_k\\
&\leq \left(\sum_{k=2}^\infty 
\mu_k^{s}\left(u_{\Omega,k} - \mu_k^{-1} u_{\partial\Omega,k} 
\right)^2\right)^{1/2}\|v\|_{{H}_{\int}^s(\Omega)},
\end{align*}
where we employed the orthogonality of the eigenfunctions $\psi_k$. From \eqref{eq:fLap_g_n} it follows that $\int_{\Omega}(-\Delta_N)^s u=-|\Omega|^{-1}\int_{\partial\Omega}\partial_\nu u$ then under the assumption that $\int_{\partial\Omega}\partial_\nu u=0$,
the operator 
$(-\Delta_N)^s$ is extendable to an operator mapping from
\[
\mathbb{N}^s(\Omega):=\Big\{u=\sum_{k=2}^\infty u_k\psi_k\in L^2(\Omega):\; 
\sum_{k=2}^\infty 
\mu_k^{s}\left(u_{\Omega,k} - \mu_k^{-1} u_{\partial\Omega,k} 
\right)^2<\infty\Big\}
\]
to ${H}_{\int}^{-s}(\Omega)$, see also Section \ref{s:app_n} where
associated boundary value problems are considered.

Similar to the Dirichlet case, we get the following properties.
\begin{proposition}
Let $(-\Delta_N)^s$ be as in Definition \ref{d:nzbc_n} then the following holds:
\begin{enumerate}
\item When $s = 1$ we obtain the standard Laplacian \eqref{eq:Lap_n}.
\item For any $u\in C^\infty(\bar\Omega)$ with $\partial_\nu u =0$ there 
holds \[(-\Delta_N)^s u = (-\Delta_{N,0})^s u,\] a.e. in  $\Omega$, i.e., we recover the Definition~\ref{d:zbc_n}.
\item For any $s\in (0,1)$ and for any $u \in C^\infty(\bar\Omega)$ with $\int_{\partial\Omega}\partial_\nu u =0$ there is the identity 
\[
(-\Delta_N)^s (-\Delta_N)^{1-s} u = -\Delta u
\] a.e. in $\Omega$.
\end{enumerate}
\end{proposition}
\begin{proof}
The first two assertions are again easy to check. To show the third one, let $u\in C^\infty(\bar\Omega)$ with $\int_{\partial\Omega}\partial_\nu u =0$
and define
\begin{equation}\label{eq:v_l}
	v_l:=\sum_{k=2}^l \left(\mu_k^{1-s} \int_\Omega u \psi_k 
	- \mu_k^{-s}\int_{\partial\Omega} \partial_\nu u \psi_k \right) 
	\psi_k.
\end{equation}
Using the orthogonality of the eigenfunctions $\psi_k$ several times, and \eqref{eq:dericlassicN}, we deduce for any $t\in[0,1]$ and for any $l,m\in\mathbb{N}$ with $l\geq m\geq 2$
\begin{align*}
	\|v_l-v_m\|_{H^{2s+t}_{\int}(\Omega)}^2&=\sum_{k=m+1}^l\mu_k^{2s+t} \left(\mu_k^{1-s} \int_\Omega u \psi_k 
	- \mu_k^{-s}\int_{\partial\Omega} \partial_\nu u \psi_k \right)^2\\
	&=\sum_{k=m+1}^l \mu_k^t\left(\mu_k \int_\Omega u \psi_k 
	- \int_{\partial\Omega} \partial_\nu u \psi_k \right)^2\\
	&=\sum_{k=m+1}^l \mu_k^t\left(\int_\Omega -\Delta u \psi_k\right)^2\leq\|\Delta u\|_{H^{t}_{\int}(\Omega)}^2,
\end{align*}
where the last term is bounded independent of $l$ and $m$ since $\Delta u\in H^1(\Omega)$. Thus, according to the Cauchy criterion, there exists a function $(-\Delta_N)^{1-s}u\in H_{\int}^{2s+1}(\Omega)\cap H^1(\Omega)$ with
\[
\lim_{l\rightarrow\infty}\|(-\Delta_N)^{1-s}u-v_l\|_{L^2(\Omega)}=0.
\]
Next, using the definition and the orthogonality of the eigenvalues and eigenfunctions once again, we deduce
\begin{align*}
\|\Delta v_l&-\Delta v_m\|_{H^{1}(\Omega)^*}=\sup_{\substack{\varphi\in H^1(\Omega)\\ \|\varphi\|_{H^1(\Omega)}=1}}\left|\int_{\Omega}\Delta(v_l-v_m)\varphi\right|\\
&=\sup_{\substack{\varphi\in H^1(\Omega)\\ \|\varphi\|_{H^1(\Omega)}=1}}\left|\sum_{k=m+1}^l\mu_k \left(\mu_k^{1-s} \int_\Omega u \psi_k 
- \mu_k^{-s}\int_{\partial\Omega} \partial_\nu u \psi_k \right)\left(\int_{\Omega}\varphi\psi_k\right)\right|\\
&\leq\sup_{\substack{\varphi\in H^1(\Omega)\\ \|\varphi\|_{H^1(\Omega)}=1}}\left(\sum_{k=m+1}^l\mu_k \left(\mu_k^{1-s} \int_\Omega u \psi_k 
- \mu_k^{-s}\int_{\partial\Omega} \partial_\nu u \psi_k \right)^2\right)^{1/2}\left(\sum_{k=m+1}^l\mu_k\left(\int_{\Omega}\varphi\psi_k\right)^2\right)^{1/2}\\
&\leq\left(\sum_{k=m+1}^l\mu_k^{1-2s} \left(\mu_k \int_\Omega u \psi_k 
- \int_{\partial\Omega} \partial_\nu u \psi_k \right)^2\right)^{1/2}=\left(\sum_{k=m+1}^l \mu_k^{1-2s}\left(\int_\Omega -\Delta u \psi_k\right)^2\right)^{1/2},
\end{align*}
where the last term is again bounded independent of $l$ and $m$ since $\Delta u\in H^1(\Omega)$. Again, due to the Cauchy criterion, there exists a function $v^*\in H^1(\Omega)^*$ with
\[
\lim_{l\rightarrow\infty}\|v^*-\Delta v_l\|_{H^{1}(\Omega)^*}=0.
\]
Moreover, for all $\varphi\in C_0^\infty(\Omega)$ we have
\begin{align*}
\int_{\Omega}(-\Delta_N)^{1-s}u\Delta\varphi=\lim_{l\rightarrow \infty}\int_{\Omega}v_l\Delta\varphi=\lim_{l\rightarrow \infty}\int_{\Omega}\Delta v_l\varphi=\int_{\Omega} v^*\varphi.
\end{align*}
Consequently, $v^*$ represents the Laplacian of $(-\Delta_N)^{1-s}u$ in the sense of distributions. In addition, we obtain $\Delta(-\Delta_N)^{1-s}u\in H^{1}(\Omega)^*$ with
\[
\lim_{l\rightarrow\infty}\|\Delta(-\Delta_N)^{1-s}u-\Delta v_l\|_{H^{1}(\Omega)^*}=0.
\]
According to \cite[Appendix A]{MR2201310} and \cite[Section 3]{gesztesy2008dirichlet} the normal derivative can be defined in a weak sense as a mapping from
\[
	\{v\in H^1(\Omega): \Delta v \in H^{1}(\Omega)^*\}\quad \text{to}\quad H^{-1/2}(\partial\Omega)
\]
by
\[
	\int_{\partial\Omega}\partial_\nu u\varphi=\int_{\Omega}\Delta u \varphi +\int_\Omega\nabla u\cdot \nabla \varphi \quad \forall \varphi\in H^1(\Omega),
\]
which is an extension of the classical normal derivative.
As a consequence, we obtain by means of the foregoing results and the definition of the eigenvalues and eigenfunctions for $k\geq2$
\begin{align*}
\int_{\partial\Omega}&\partial_\nu (-\Delta_N)^{1-s}u\psi_k=\int_{\Omega}\Delta (-\Delta_N)^{1-s}u \psi_k +\int_\Omega\nabla (-\Delta_N)^{1-s}u\cdot \nabla \psi_k\\
&=\lim_{l\rightarrow\infty}\int_{\Omega}\Delta v_l \psi_k +\mu_k\int_\Omega (-\Delta_N)^{1-s}u \psi_k\\
&=-\mu_k \left(\mu_k^{1-s} \int_\Omega u \psi_k 
- \mu_k^{-s}\int_{\partial\Omega} \partial_\nu u \psi_k \right)+\mu_k\left(\mu_k^{1-s} \int_\Omega u \psi_k 
- \mu_k^{-s}\int_{\partial\Omega} \partial_\nu u \psi_k \right)\\
&=0.
\end{align*}
For $k=1$ we get
\[
	\int_{\partial\Omega}\partial_\nu (-\Delta_N)^{1-s}u\psi_1=\int_{\Omega}\Delta (-\Delta_N)^{1-s}u \psi_1=\lim_{l\rightarrow\infty}\int_{\Omega}\Delta v_l \psi_1=0.
\]
The above observations allow us to apply $(-\Delta_N)^s$ to $(-\Delta_N)^{1-s}u$.
For that purpose, we define
\[
w_l=\sum_{k=1}^l \mu_k^{s}\left( \int_\Omega (-\Delta_N)^{1-s}u \psi_k \right) 
\psi_k.
\]
As before we deduce for any $l,m\in\mathbb{N}$ with $l\geq m$
\begin{align*}
\|w_l-w_m\|_{L^2(\Omega)}^2&=\sum_{k=m+1}^l\mu_k^{2s}\left( \int_\Omega (-\Delta_N)^{1-s}u \psi_k \right)^2\leq\|(-\Delta_N)^{1-s} u\|_{H^{2s}_{\int}(\Omega)}^2,
\end{align*}
where the last term is bounded independent of $l$ and $m$ since $(-\Delta_N)^{1-s}u\in H_{\int}^{2s+1}(\Omega)\cap H^1(\Omega)$. As a consequence, due to the Cauchy criterion, the limit $(-\Delta_N)^{s}(-\Delta_N)^{1-s}u:=\lim_{l\rightarrow\infty}w_l$ exists in $L^2(\Omega).$
Finally, according to the orthogonality of the eigenfunctions $\psi_k$ and \eqref{eq:dericlassicN}, we obtain for $k>1$
\begin{align*}
\int_{\Omega}(-\Delta_N)^s (-\Delta_N)^{1-s} u \psi_k&= \mu_k^{s}
\int_{\Omega} (-\Delta_N)^{1-s}u \psi_k = \mu_k^{s}\left(\mu_k^{1-s} \int_\Omega u \psi_k 
- \mu_k^{-s}\int_{\partial\Omega} \partial_\nu u \psi_k \right)\\
&= \mu_k \int_\Omega u \psi_k - \int_{\partial\Omega} \partial_\nu u \psi_k= \int_{\Omega} -\Delta u \psi_k.
\end{align*}
Moreover, we deduce
\[
	\int_{\Omega}(-\Delta_N)^s (-\Delta_N)^{1-s} u \psi_1=-\int_{\partial\Omega}\partial_\nu (-\Delta_N)^{1-s}u\psi_1=0=-\int_{\partial\Omega}\partial_\nu u\psi_1=\int_{\Omega} -\Delta u \psi_1.
\]
Since $\{\psi_k\}_{k\in\mathbb{N}}$ represents an orthonormal basis of $L^2(\Omega)$, we conclude the result.
\end{proof}

%%%%%%%%%%%%%%%%%%%%%%%%%%%%%%%%%%%%%%%%%%%%%%%%%%%%%%%%%%%%%%%%%%%%%%%%%%%%%%%%%%%%%%%%%%%%%%%%%%%%%%%%%%%%%%%%%%%%%%%%%%
\section{Trace Theorems and Integration-by-parts Type Formulas}\label{s:Ntrthm}
%%%%%%%%%%%%%%%%%%%%%%%%%%%%%%%%%%%%%%%%%%%%%%%%%%%%%%%%%%%%%%%%%%%%%%%%%%%%%%%%%%%%%%%%%%%%%%%%%%%%%%%%%%%%%%%%%%%%%%%%%%

The purpose of this section is to state the Neumann trace space 
for $H^2(\Omega) \cap H^1_0(\Omega)$ and the Dirichlet trace space for 
functions belonging to $\{v\in H^2(\Omega) \;:\;\partial_\nu v =0 \text{ a.e. on }\partial\Omega\}$.

We begin by 
introducing 
the reflexive Banach space $N^{s}(\partial\Omega)$ with $s\in[0,\frac12]$ which is 
defined as
\begin{equation}\label{eq:spaceN}
 N^{s}(\partial\Omega) := \{ g \in L^2(\partial\Omega) \;:\; g\nu_k \in H^{s}(\partial\Omega), 
 \ 1 \le k \le n \} 
\end{equation}
with norm
\begin{equation}\label{eq:normN}
 \|g\|_{N^{s}(\partial\Omega)} = \sum_{k=1}^n \|g 
 \nu_k\|_{H^{s}(\partial\Omega)} . 
\end{equation}
Due to the fact that
\[
	\|g\|_{L^2(\partial\Omega)}\sim \|g\|_{N^0(\partial\Omega)}\quad\forall g\in L^2(\partial\Omega),
\]
we notice that
	\begin{equation}\label{eq:equivalence1}
N^{s}(\partial\Omega)=[{N^{1/2}(\partial\Omega)},{L^{2}(\partial\Omega)}]_{1-2s},
\end{equation}
which can be deduced by classical results of real interpolation.

For $s=\frac12$ we state the following trace theorem for the Neumann trace operator.

\begin{lemma}\label{l:Nspace}
Let $n \ge 2$ and $\Omega$ be a bounded Lipschitz domain. Then the Neumann trace operator $\partial_\nu$
\[
 \partial_\nu : H^1_0(\Omega)\cap H^2(\Omega) \rightarrow N^{1/2}(\partial\Omega)
\]
is well-defined, linear, bounded, onto, and with a linear, bounded right inverse. 
Additionally, the null space of $\partial_\nu$ is $H^2_0(\Omega)$, the closure of 
$C_0^\infty(\Omega)$ in $H^2(\Omega)$. 
\end{lemma}
\begin{proof}
See Lemma 6.3 of \cite{MR2788354}.
\end{proof}

\begin{remark}[Relation between $N^{s}(\partial\Omega)$ and $H^{s}(\partial\Omega)$ for $s \in {[0,\frac12]}$]
\label{r:Nspace}  
If $\Omega$ is of class $C^{1,r}$ with $r > 1/2$ then $N^{s}(\partial\Omega) = H^{s}(\partial\Omega)$ for $s\in[0,\frac12]$, see \cite[Lemma 6.2]{MR2788354}. 
\end{remark}

Next, we state an integration-by-parts type formula which relates $(-\Delta_D)^s$ to $(-\Delta_{D,0})^s$. In order to do so, we need to assume that the domain $\Omega$ is quasi-convex, see \cite[Definition 8.9]{MR2788354}. The latter is a subset of bounded Lipschitz domains which is locally \emph{almost convex}. For a precise definition of an \emph{almost convex} domain we refer to \cite[Definition 8.4]{MR2788354}. In the class of bounded Lipschitz domains the following sequence holds (see \cite{MR2788354}):
\[
\mbox{convex} \implies \mbox{UEBC} \implies \mbox{LEBC} \implies \mbox{almost convex} \implies \mbox{quasi-convex}  
\]
where UEBC and LEBC stands for bounded Lipschitz domains which fulfill the uniform exterior ball condition and local exterior ball condition, respectively. We further remark that a bounded Lipschitz domain which fulfills UEBC is also known as semiconvex domain \cite[Theorem~3.9]{MR2593333}.
\begin{theorem}[Dirichlet: integration-by-parts formula]\label{theorem:Dint}
	Let $\Omega$ be a bounded quasi-convex domain. Moreover, let $u \in \mathbb{D}^{2s}(\Omega)$ with $u|_{\partial\Omega}\in N^{1/2}(\partial\Omega)^*$ and $v \in \mathbb{H}^{2s}(\Omega)$. Then the following integration-by-parts formula holds
	\begin{align*}
	\int_{\Omega}(-\Delta_D)^suv&=
	\int_\Omega u (-\Delta_{D,0})^sv+\int_{\partial\Omega}u \partial_\nu w_v,
	\end{align*}	
	where $w_v \in \mathbb{H}^{2}(\Omega)$ is defined as the solution to
	\begin{align}
	(-\Delta_{D,0})^{1-s}w_v=v\quad \text{in } \Omega,\quad w_v=0\quad\text{on 
	}\partial \Omega.\label{eq:wv}
	\end{align}
\end{theorem}
\begin{proof}
Let us define
\[
v_l:=\sum_{k=1}^l\left(\lambda_k^{s} \int_\Omega u \varphi_k 
+ \lambda_k^{s-1} \int_{\partial\Omega} u \partial_\nu \varphi_k \right) 
\varphi_k.
\]
Using the orthogonality of the eigenfunctions $\varphi_k$, we obtain for any $l,m\in \mathbb{N}$ with $l\geq m$
\begin{align*}
	\|v_l-v_m\|_{L^2(\Omega)}^2&=\sum_{m+1}^l \left(\lambda_k^{s} \int_\Omega u \varphi_k 
	+ \lambda_k^{s-1} \int_{\partial\Omega} u \partial_\nu \varphi_k \right)^2\\
	&=\sum_{m+1}^l \lambda_k^{2s}\left(\int_\Omega u \varphi_k 
	+ \lambda_k^{-1} \int_{\partial\Omega} u \partial_\nu \varphi_k \right)^2\leq \|u\|_{\mathbb{D}^{2s}(\Omega)}^2,
\end{align*}
where the last term is bounded independent of $l$ and $m$ since $u\in \mathbb{D}^{2s}(\Omega)$. As a consequence, by means of the Cauchy criterion, the limit $(-\Delta_D)^{s} u:= \lim_{l\rightarrow\infty}v_l$ exists in $L^2(\Omega)$. In the same manner, we obtain for $v\in\mathbb{H}^{2s}(\Omega)$ that the limit
\[
	(-\Delta_{D,0})^sv:=\lim_{l\rightarrow\infty} \sum_{k=1}^l\lambda_k^s\left(\int_{\Omega}v\varphi_k\right)\varphi_k
\]
exits is $L^2(\Omega)$. Moreover, if we set \[w_{v,k}:=\lambda_k^{s-1}\left(\int_{\Omega}v \varphi_k\right),\]
we can conclude that the solution $w_v$ to \eqref{eq:wv} is given by
\[
	w_v=\lim_{l\rightarrow\infty}\sum_{k=1}^lw_{v,k}\varphi_k=\lim_{l\rightarrow\infty}\sum_{k=1}^l\lambda_{k}^{s-1}\left(\int_\Omega v\varphi_k\right)\varphi_k.
\]
With similar arguments as above, it is straightforward to verify that $w_v\in \mathbb{H}^2(\Omega)$, $\Delta w_v \in L^2(\Omega)$ and
\begin{equation}\label{eq:w_v1}
	\|\Delta w_v\|_{L^2(\Omega)}=\|w_v\|_{\mathbb{H}^2(\Omega)}
\end{equation}
since $v\in\mathbb{H}^{2s}(\Omega)$. 
Combining the last observations yields
\begin{align*}
\int_{\Omega} (-\Delta_D)^suv&=\lim_{l\rightarrow\infty}\sum_{k=1}^l \left(\lambda_k^{s} \int_\Omega u \varphi_k + \lambda_k^{s-1} \int_{\partial\Omega} u \partial_\nu \varphi_k \right) 
\left(\int_{\Omega} v\varphi_k\right)\\
&=\lim_{l\rightarrow\infty}\int_\Omega u \sum_{k=1}^l\lambda_k^{s}\left(\int_\Omega v\varphi_k\right)\varphi_k + \lim_{l\rightarrow\infty} \int_{\partial\Omega} u \partial_\nu \sum_{k=1}^l\lambda_k^{s-1} \left(\int_{\Omega} v\varphi_k\right)\varphi_k\\
&=\int_\Omega u (-\Delta_{D,0})^sv+\int_{\partial\Omega}u \partial_\nu w_v+\lim_{l\rightarrow\infty} \int_{\partial\Omega} u \partial_\nu \left(\sum_{k=1}^lw_{v,k}\varphi_k-w_v\right).
\end{align*}
The proof is complete, once we have shown that
\[
	\lim_{l\rightarrow\infty} \int_{\partial\Omega} u \partial_\nu \left(\sum_{k=1}^lw_{v,k}\varphi_k-w_v\right)=0.
\]
To this end, we notice that in quasi-convex domains, for $\Delta w_v \in L^2(\Omega)$, the solution $w_v$ of \eqref{eq:wv} belongs to $H^2(\Omega)\cap H^1_0(\Omega)$ and fulfills
\begin{equation}\label{eq:w_v2}
	\|w_v\|_{H^2(\Omega)}\leq c \|\Delta w_v\|_{L^2(\Omega)},
\end{equation}
see e.g. \cite[Theorem~10.4]{MR2788354}. According to \cite[Corollary 6.5]{MR2788354}, the duality pairing between $N^{1/2}(\partial\Omega)$ and $N^{1/2}(\partial\Omega)^*$ is compatible with the natural integral pairing in $L^2(\partial\Omega)$. Consequently, using Lemma \ref{l:Nspace}, \eqref{eq:w_v2} and \eqref{eq:w_v1}, we obtain
\begin{align*}
	\left|\int_{\partial\Omega} u \partial_\nu \left(\sum_{k=1}^lw_{v,k}\varphi_k-w_v\right)\right|&\leq \|u\|_{N^{1/2}(\partial\Omega)^*}\|\partial_\nu \left(\sum_{k=1}^lw_{v,k}\varphi_k-w_v\right)\|_{N^{1/2}(\partial\Omega)}\\
	&\leq c\|u\|_{N^{1/2}(\partial\Omega)^*}\|\sum_{k=1}^lw_{v,k}\varphi_k-w_v\|_{\mathbb{H}^2(\Omega)}.
\end{align*}
Thus, the assertion is proved since $\lim_{l\to\infty}\|\sum_{k=1}^lw_{v,k}\varphi_k-w_v\|_{\mathbb{H}^2(\Omega)}=0$ as shown above.
\end{proof}

\begin{remark}\begin{enumerate}
	\item If in Theorem~\ref{theorem:Dint} we let $u\in H^{1/2}(\partial\Omega)$ then we do not need to assume quasi-convexity, bounded Lipschitz domains are sufficient. Following e.g. \cite{MR2788354}, it may even be possible to further relax the regularity requirements for $u$ on the boundary in this case.
	\item If in Theorem~\ref{theorem:Dint} we let $u\in \mathbb{D}^2(\Omega)$ then $\Delta u \in L^2(\Omega)$. As a consequence, according to \cite[Theorem 6.4]{MR2788354}, we obtain $u\in N^{1/2}(\partial\Omega)^*$.
	\end{enumerate}
\end{remark}

We continue by introducing the space $N^{s}(\partial\Omega)$ with $s\in[1,\frac32]$,
\[
N^{s}(\partial\Omega) := \{ g \in H^1(\partial\Omega) \;:\; \nabla_{tan}g\in 
(H^{s-1}(\partial\Omega))^n\}. 
\]
This space can be endowed with the norm
\[
\|g\|_{N^{s}(\partial\Omega)} = \|g 
\|_{L^2(\partial\Omega)}+\|\nabla_{tan}g\|_{(H^{s-1}(\partial\Omega))^n}.
\]
Here $\nabla_{tan} g = \left(\sum_{k=1}^n \nu_k  \frac{\partial g}{\partial \tau_{k,j}}\right)_{1\le j \le n}$ with $\frac{\partial}{\partial \tau_{k,j}} = \nu_k \frac{\partial}{\partial x_j} - \nu_j \frac{\partial}{\partial x_k}$.

Similarly to the explanations above, we obtain by the fact that
	\[
	\|g\|_{H^{1}(\partial\Omega)}\sim \|g\|_{N^1(\partial\Omega)}\quad\forall g\in H^1(\partial\Omega),
	\]
	the following characterization of the intermediate spaces
	\begin{equation}\label{eq:equivalence2}
	N^{s}(\partial\Omega)=[{N^{3/2}(\partial\Omega)},{H^{1}(\partial\Omega)}]_{3-2s},
	\end{equation}
	which is due to classical results of real interpolation.
	
	For $s=\frac32$ we have the following result for the Dirichlet trace operator.

\begin{lemma}\label{l:NspaceD}
	Let $n \ge 2$ and $\Omega$ be a bounded Lipschitz domain. Then the Dirichlet trace 
	operator $\gamma_D$
	\[
	\gamma_D : \{v\in H^2(\Omega) \;:\;\partial_\nu u =0 \text{ a.e. on }\partial\Omega\} \rightarrow 
	N^{3/2}(\partial\Omega)
	\]
	is well-defined, linear, bounded, onto, and with a linear, bounded right 
	inverse. 
	Additionally, the null space of $\gamma_D$ is $H^2_0(\Omega)$, the closure 
	of 
	$C_0^\infty(\Omega)$ in $H^2(\Omega)$. 
\end{lemma}
\begin{proof}
	See Lemma 6.9 of \cite{MR2788354}.
\end{proof}

\begin{remark}[Relation between $N^{s}(\partial\Omega)$ and $H^{s}(\partial\Omega)$ for $s\in {[1,\frac32]}$]
\label{r:N32space}
If $\Omega$ is of class $C^{1,r}$ with $r > 1/2$ then $N^{s}(\partial\Omega) = H^{s}(\partial\Omega)$ for $s\in[1,\frac32]$, see \cite[Lemma 6.8]{MR2788354}. 
\end{remark}

As for the Dirichlet fractional Laplacian, we are able to state an integration-by-parts type formula which relates $(-\Delta_N)^s$ to $(-\Delta_{N,0})^s$.
	\begin{theorem}[Neumann: integration-by-parts formula]
	\label{theorem:Nint}
		Let $\Omega$ be a bounded quasi-convex domain. Moreover, let $u \in \mathbb{N}^{2s}(\Omega)$ with $\partial_\nu u\in N^{3/2}(\partial\Omega)^*$ and $v \in H^{2s}_{\int}(\Omega)$. Then the following integration-by-parts formula holds
		\begin{align*}
		\int_{\Omega}(-\Delta_N)^suv&=
		\int_\Omega u (-\Delta_{N,0})^sv-\int_{\partial\Omega}\partial_\nu u w_v,
		\end{align*}	
		where $w_v \in H^{2}_{\int}(\Omega)$ is defined as the solution to
		\begin{align}\notag
		(-\Delta_{N,0})^{1-s}w_v=v\quad \text{in } \Omega,\quad \partial_{\nu}w_v=0\quad\text{on 
		}\partial \Omega.
		\end{align}
\end{theorem}
\begin{proof}
The proof is almost a word-by-word repetition of the proof of Theorem \ref{theorem:Dint}. In contrast it is crucial to show that
\[
\lim_{l\rightarrow\infty} \int_{\partial\Omega} \partial_\nu u \left(\sum_{k=2}^lw_{v,k}\psi_k-w_v\right)=0
\]
with $w_{v,k}:=\mu_k^{s-1}\int_{\Omega}v\psi_k$. It is again straightforward to verify that $w_v\in H^2_{\int}(\Omega)$, $\Delta w_v \in L^2(\Omega)$ and $\|\Delta w_v\|_{L^2(\Omega)}=\|w_v\|_{H^2_{\int}(\Omega)}$. As a consequence, according to \cite[Theorem 10.8]{MR2788354}, the solution $w_v$ belongs to $\{v\in H^2(\Omega) :\partial_\nu u =0 \text{ a.e. on }\partial\Omega\}$ and fulfills
\[
	\|w_v\|_{H^2(\Omega)}\leq c\|\Delta w_v\|_{L^2(\Omega)}.
\]
According to \cite[Corollary 6.12]{MR2788354}, the duality pairing between $N^{3/2}(\partial\Omega)$ and $N^{3/2}(\partial\Omega)^*$ is compatible with the natural integral pairing in $L^2(\partial\Omega)$. Using this in combination with Lemma \ref{l:NspaceD} and the foregoing results, we obtain
\begin{align*}
\left|\int_{\partial\Omega} \partial_\nu u \left(\sum_{k=2}^lw_{v,k}\psi_k-w_v\right)\right|&\leq \|\partial_\nu u\|_{N^{3/2}(\partial\Omega)^*}\|\sum_{k=2}^lw_{v,k}\psi_k-w_v\|_{N^{3/2}(\partial\Omega)}\\
&\leq c\|\partial_\nu u\|_{N^{3/2}(\partial\Omega)^*}\|\sum_{k=2}^lw_{v,k}\psi_k-w_v\|_{{H}^2_{\int}(\Omega)}.
\end{align*}
Again, the assertion is proved since $\lim_{l\to\infty}\|\sum_{k=2}^lw_{v,k}\psi_k-w_v\|_{{H}^2_{\int}(\Omega)}=0$ as shown above.
\end{proof}
\begin{remark}\begin{enumerate}
		\item If in Theorem~\ref{theorem:Nint} we let $\partial_\nu u\in H^{-1/2}(\partial\Omega)$ then we do not need to assume quasi-convexity , bounded Lipschitz domains are again sufficient.
		\item If in Theorem~\ref{theorem:Nint} we let $u\in \mathbb{N}^2(\Omega)$ and $\partial_{\nu}u$ Lebesgue measurable then $\Delta u \in L^2(\Omega)$. Consequently, using \cite[Theorem 6.10]{MR2788354}, we obtain $\partial_\nu u\in N^{3/2}(\partial\Omega)^*$.
	\end{enumerate}
\end{remark}

%%%%%%%%%%%%%%%%%%%%%%%%%%%%%%%%%%%%%%%%%%%%%%%%%%%%%%%%%%%%%%%%%%%%%%%%%%%%%%%%%%%%%%%%%%%%%%%%%%%%%%%%%%%%%%%%%%%%%%%%%%
\section{Application I: Fractional Equation with Dirichlet Boundary Condition}\label{s:app}
%%%%%%%%%%%%%%%%%%%%%%%%%%%%%%%%%%%%%%%%%%%%%%%%%%%%%%%%%%%%%%%%%%%%%%%%%%%%%%%%%%%%%%%%%%%%%%%%%%%%%%%%%%%%%%%%%%%%%%%%%%

We next apply our definition in \eqref{eq:fLap_g} to \eqref{eq:fracLap}. In 
order to impose the boundary condition $u = g$ on $\partial\Omega$, we use 
the standard lifting argument, i.e., given $g \in 
\mathbb{D}^{s-\frac{1}{2}}(\partial\Omega)$
with
\[
\mathbb{D}^{s-\frac{1}{2}}(\partial\Omega):=\begin{cases} N^{\frac12-s}(\partial\Omega)^*&\text{for }s\in[0,\frac12)\\
L^2(\partial\Omega)&\text{for }s=\frac12\\
H^{s-\frac12}(\partial\Omega)&\text{for }s\in(\frac12,1]
\end{cases},
\]
we construct $v \in \mathbb{D}^s(\Omega)$ solving 
\begin{align}\label{eq:v}
\begin{aligned}
(-\Delta_D)^s v &= 0  \quad \mbox{in } \Omega, \\
v &= g  \quad \mbox{on } \partial\Omega ,
\end{aligned}
\end{align}
and given $f \in \mathbb{H}^{-s}(\Omega)$, $w \in \mathbb{H}^s(\Omega)$ solves
\begin{align}\label{eq:w}
\begin{aligned}
 (-\Delta_D)^s w &= f  \quad \mbox{in } \Omega     ,    \\
             w &= 0  \quad \mbox{on } \partial\Omega ,
\end{aligned}
\end{align}
then $u = w+v$. 
Notice that in \eqref{eq:w} $(-\Delta_D)^s = (-\Delta_{D,0})^s$ by Proposition \ref{p:presvD} and density. Study of 
\eqref{eq:w} has been the focal point of 
several recent works \cite{LCaffarelli_LSilvestre_2007a, 
PRStinga_JLTorrea_2010a,CBrandle_EColorado_APablo_USSanchez_2013a,XCabre_JTan_2010a,
 ACapella_JDavila_LDupaigne_YSire_2011a} and can be realized by using the 
Caffarelli-Silvestre extension or the Stinga-Torrea extension \cite{LCaffarelli_LSilvestre_2007a, 
PRStinga_JLTorrea_2010a}, see for instance 
\cite{RHNochetto_EOtarola_AJSalgado_2014a}. 

On the other hand, at the first glance, \eqref{eq:v} seems as complicated as 
the original problem \eqref{eq:fracLap}. However, we will show that 
\eqref{eq:v} is equivalent to solving a standard Laplace problem with nonzero 
boundary conditions  
\begin{equation}\label{eq:v_ver}
 -\Delta v = 0 \quad \mbox{in } \Omega, \quad v = g \quad \mbox{on } \partial\Omega, 
\end{equation}
in the so-called very-weak form \cite{LionsMagenes1968, 
	MR2084239,CasasRaymond2006,MR3070527,MR3432846,NUM:NUM22057,apelnicaisepfefferer:2015-2extended}
	 or in the classical 
weak form if the regularity of the boundary datum guarantees its 
well-posedness. 

We start with introducing the very weak form of \eqref{eq:v_ver}. Given $g\in 
N^{\frac{1}{2}}(\partial\Omega)^*$, we are seeking a function $v\in 
L^2(\Omega)$ fulfilling 
\begin{equation}\label{eq:v_veryweak}
	\int_{\Omega}v(-\Delta)\varphi=-\int_{\partial\Omega}g\partial_\nu 
	\varphi 
	\quad
	\forall \varphi\in V:=H^1_0(\Omega)\cap H^2(\Omega).
\end{equation}
Next, we show existence and regularity results for the very 
weak solution of \eqref{eq:v_ver}.

\begin{lemma}
\label{l:v_exist}
    Let $\Omega$ be a bounded, quasi-convex domain. 
	For any $g\in N^{\frac12}(\partial\Omega)^*$,
	there exists an unique very weak 
	solution $v\in L^2(\Omega)$ of \eqref{eq:v_ver}. For more regular boundary 
	data $g\in \mathbb{D}^{s-\frac{1}{2}}(\partial\Omega)$ with $s\in [0,1]$, the 
	solution 
	belongs to $H^s(\Omega)$ and admits the a priori estimate
	\begin{equation}\label{eq:apriorivwD}
		\|v\|_{H^s(\Omega)}\leq c 
		\|g\|_{\mathbb{D}^{s-\frac{1}{2}}(\partial\Omega)}.
	\end{equation}
	Moreover, if s=1 then the very weak solution is actually a weak solution.
\end{lemma}
\begin{remark}
	Notice that owing to Remark~\ref{r:Nspace}, when $\Omega$ is $C^{1,r}$ with 
	$r > 1/2$ we have $N^{\frac{1}{2}}(\partial\Omega)^* = 
	H^{-\frac{1}{2}}(\partial\Omega)$ and thus 
	$\mathbb{D}^{s-\frac{1}{2}}(\partial\Omega) = 
	H^{s-\frac{1}{2}}(\partial\Omega)$. Moreover, by employing similar arguments, in combination with \cite{MR2788354}, the results of Lemma \ref{l:v_exist} can be extended to general Lipschitz domains at least for $s\in[\frac12,1]$.
\end{remark}
\begin{proof}
	The idea of the existence and uniqueness proof of a solution
	to \eqref{eq:v_veryweak} is based on applying the 
	Babu\v{s}ka-Lax-Milgram theorem. This is already outlined in the proof of 
	Lemma 2.3 
	in \cite{NUM:NUM22057}. However, in that reference, the focus was on two 
	dimensional polygonal domains. Since we are working in $n$ space dimensions 
	with different assumptions on the boundary, we 
	present the proof again, also for the convenience of the reader. We also refer to \cite{MR2788354} for related results.

	First, we notice that the bilinear form associated to \eqref{eq:v_veryweak} 
	is obviously bounded on $L^2(\Omega)\times V$. In order to show the inf-sup 
	conditions, we use the isomorphism
	\[
	\Delta\varphi\in L^2(\Omega),\ \varphi|_{\partial\Omega}=0\quad 
	\Leftrightarrow \quad \varphi\in V,
	\]
	which is valid under the present assumptions on the domain according to 
	\cite[Theorem~10.4]{MR2788354}. 
	A norm in $V$ is 
	given by $\|\varphi\|_V=\|\Delta\varphi\|_{L^2(\Omega)}$ due to the 
	standard a priori estimate 
	$\|\varphi\|_{H^2(\Omega)}\leq c\|\Delta\varphi\|_{L^2(\Omega)}$. 
	Then by taking 
	$v=-\Delta\varphi/\|\Delta\varphi\|_{L^2(\Omega)}\in 
	L^2(\Omega)$, we deduce
	\[
	\sup_{\substack{v\in L^2(\Omega)\\ 
			\|v\|_{L^2(\Omega)}=1}}|(v,-\Delta\varphi)_{L^2(\Omega)}|
	\geq\frac{|(\Delta\varphi,\Delta\varphi)_{L^2(\Omega)}|}
	{\|\Delta\varphi\|_{L^2(\Omega)}}=\|\Delta\varphi\|_{L^2(\Omega)}
	=\|\varphi\|_{V}.
	\]
	If we choose $\varphi\in V$ as the solution of 
	$-\Delta\varphi=v/\|v\|_{L^2(\Omega)}$ with some $v\in L^2(\Omega)$ 
	then we obtain
	\[
	\sup_{\substack{\varphi\in V\\ 
			\|\varphi\|_{V}=1}}|(v,-\Delta\varphi)_{L^2(\Omega)}|
	\geq\frac{|(v,v)_{L^2(\Omega)}|}%
	{\|v\|_{L^2(\Omega)}}=\|v\|_{L^2(\Omega)}.
	\]
	It remains to check that the right hand side of \eqref{eq:v_veryweak} 
	defines a linear functional on $V$ for any $g\in 
	N^{\frac12}(\partial\Omega)^*$. 
	In view of Lemma~\ref{l:Nspace} we have that 
	 \begin{equation}\label{eq:apri}
	  \left| \int_{\partial\Omega} g \partial_\nu \varphi \right| 
	   \le \|g\|_{N^{\frac12}(\partial\Omega)^*}
	       \| \partial_\nu \varphi \|_{N^{\frac12}(\partial\Omega)}
	   \le C \|g\|_{N^{\frac12}(\partial\Omega)^*} \|\varphi\|_V .   
	 \end{equation}
	Thus, all the requirements of the Babu\v{s}ka-Lax-Milgram theorem are 
	fulfilled and we can deduce the existence of a unique solution in 
	$L^2(\Omega)$ for any 
	Dirichlet boundary datum $g\in N^{\frac12}(\partial\Omega)^*$.

	The 
	a priori 
	estimate in that case is a simple consequence of the above shown 
	inf-sup condition combined with \eqref{eq:v_veryweak} and \eqref{eq:apri}. 	Indeed, 
	\begin{equation}\label{eq:Daprior1}
	 \|v\|_{L^2(\Omega)} \le \sup_{\substack{\varphi\in V\\ 
			\|\varphi\|_{V}=1}}|(v,-\Delta \varphi)_{L^2(\Omega)}|
			= \sup_{\substack{\varphi\in V\\ 
			\|\varphi\|_{V}=1}}\left| \int_{\partial\Omega} g \partial_\nu \varphi \right|
			\le \|g\|_{N^{\frac{1}{2}}(\partial\Omega)^*} .
	\end{equation}
	Moreover, according to \cite[Theorem 5.3]{MR2788354} there holds
	\begin{equation}\label{eq:Dapriori2}
	\|v\|_{H^{\frac12}(\Omega)}\leq c\|g\|_{L^{2}(\partial\Omega)}.
	\end{equation}	
	Next, we show that for any $g\in H^{1/2}(\partial\Omega)$ the very weak 
	solution belongs to 
	$H^1(\Omega)$ and represents actually a weak solution. For the weak formulation of problem \eqref{eq:v_ver} it is 
	classical to show that for those data there is a unique weak solution in 
	$H^1(\Omega)$ 
	fulfilling the a priori estimate
	\begin{equation}\label{eq:Dapriori3}
	\|v\|_{H^1(\Omega)}\leq c\|g\|_{H^{1/2}(\partial\Omega)}.
	\end{equation}
	According to the integration by parts formula in
	\cite{MR937473} 
	\[
	 (\partial_\nu \varphi, \chi)_{\partial\Omega}
	 = (\nabla \varphi, \nabla \chi)_\Omega 
	  + (\Delta \varphi, \chi)_{\Omega} \quad 
	  \forall \varphi \in V , \quad \forall \chi \in H^1(\Omega),
	\]	
	we can check that any weak solution represents a very weak solution. Just 
	set $\chi = v$ and use $(\nabla \varphi, \nabla v) = 0$. Due to the uniqueness of both, the weak and the very weak solution, 
	they must 
	coincide. Finally, by real interpolation in Sobolev spaces, we can conclude, 
	according to \eqref{eq:Daprior1}--\eqref{eq:Dapriori3} and \eqref{eq:equivalence1}, the existence of a solution in $H^s(\Omega)$ for any boundary datum $g\in 
	\mathbb{D}^{s-\frac{1}{2}}(\partial\Omega)$, and the validity of the a priori estimate, 
	which ends the proof.
\end{proof}

Next we show the uniqueness of the fractional problem for $v$ solution to \eqref{eq:v}. We shall use this result, in combination with Lemma~\ref{l:v_exist}, to show the existence of a solution to \eqref{eq:v}. 
\begin{lemma}\label{l:v_uniq}
A solution $v\in \mathbb{D}^s(\Omega)$ to \eqref{eq:v} is unique. 
\end{lemma}
\begin{proof}
Since \eqref{eq:v} is linear, it is sufficient to show that when $g \equiv 0$ then $v \equiv 0$. The function $v\in\mathbb{D}^s(\Omega)$, solution to \eqref{eq:v} with $g=0$, fulfills
\[
 \sum_{k=1}^\infty \lambda_k^s \int_\Omega v \varphi_k \int_\Omega \phi 
 \varphi_k = 0 \quad \forall \phi\in\mathbb{H}^s(\Omega).
\]
Setting $\phi = v$, we arrive at the asserted result.
\end{proof}

\begin{theorem}\label{t:veqver}
Let the assumptions of Lemma~\ref{l:v_exist} hold. Then solving problem 
\eqref{eq:v} is equivalent to solving 
problem \eqref{eq:v_ver} in the very weak sense. As a consequence, the results of Lemma \ref{l:v_exist} are valid for the solution of the fractional problem \eqref{eq:v}. 
\end{theorem}
\begin{proof} 
Since both \eqref{eq:v} and \eqref{eq:v_ver} have unique solutions, it is sufficient to show that the solution $v\in \mathbb{D}^s(\Omega)$ to \eqref{eq:v} solves \eqref{eq:v_ver} in the very weak sense. 
The solution $v\in \mathbb{D}^s(\Omega)$ to \eqref{eq:v} fulfills
\[
	\sum_{j=1}^\infty\left(\lambda_j^s\int_{\Omega}v\varphi_j+\lambda_j^{s-1}\int_{\partial\Omega}g\partial_\nu\varphi_j\right)\int_{\Omega}\phi\varphi_j=0\quad \forall \phi\in \mathbb{H}^s(\Omega).
\]
Taking an arbitrary eigenfunction $\varphi_k$ as a test function, and employing the orthogonality of the eigenfunctions in $L^2(\Omega)$, we obtain
\[
  0 = \lambda_k^s \int_\Omega v \varphi_k + \lambda_k^{s-1} 
  \int_{\partial\Omega} g \partial_\nu \varphi_k 
  = \lambda_k^{s-1} \left( \lambda_k\int_\Omega v \varphi_k + 
  \int_{\partial\Omega} g \partial_\nu \varphi_k \right) .
\]
Since $\lambda_k > 0$ and $-\Delta \varphi_k = \lambda_k \varphi_k$, we have arrived at
\[
 \int_\Omega v (-\Delta)\varphi_k = - \int_{\partial\Omega} g 
 \partial_\nu 
 \varphi_k. 
\] 
Since a basis of $V:=\operatorname{dom}(-\Delta_{D,0})$ is given by the 
eigenfunctions $\varphi_k$ we have shown that $v\in \mathbb{D}^s(\Omega)$ 
solves \eqref{eq:v_ver}.
This concludes the proof.
\end{proof}

\begin{theorem}[Existence and uniqueness]\label{t:exist}
Let $\Omega$ be a bounded, quasi-convex domain. If $f \in \mathbb{H}^{-s}(\Omega)$, $g\in \mathbb{D}^{s-\frac{1}{2}}(\partial\Omega)$ then \eqref{eq:fracLap} has a unique solution $u \in H^s(\Omega)$ which satisfies 
\begin{equation}\label{eq:contdepdata}
 \|u\|_{H^s(\Omega)} \le C \left( \|f\|_{\mathbb{H}^{-s}(\Omega)} + \|g\|_{\mathbb{D}^{s-\frac{1}{2}}(\partial\Omega)} \right) ,
\end{equation}
where $C$ is a positive constant independent of $u, f$, and $g$. 
\end{theorem}
\begin{proof}
Notice that solving \eqref{eq:fracLap} for $u$ is equivalent to solving \eqref{eq:v} and \eqref{eq:w} for $v$ and $w$, respectively. Then $u = w+v$. The existence and uniqueness of $w \in \mathbb{H}^s(\Omega)$ for Lipschitz domains is due to \cite[Theorem~2.5]{MR3489634}.
The existence and uniqueness of $v \in H^s(\Omega)$ is given by Theorem~\ref{t:veqver} which says that \eqref{eq:v} is equivalent to \eqref{eq:v_ver} such that the results of Lemma~\ref{l:v_exist} apply. 
Finally, using the triangle inequality we obtain
\[
 \|u\|_{H^s(\Omega)} \le \|w\|_{H^s(\Omega)} + \|v\|_{H^s(\Omega)} . 
\]
From Theorem~\ref{t:veqver} and Lemma~\ref{l:v_exist} we know that $\|v\|_{H^s(\Omega)} \le C 
\|g\|_{\mathbb{D}^{s-\frac12}(\partial\Omega)}$. It remains to estimate 
$\|w\|_{H^s(\Omega)}$. Using Proposition~\ref{l:HsBHs} we obtain 
\[
 \|w\|_{H^s(\Omega)} \le C \|w\|_{\mathbb{H}^s(\Omega)} ,
\]
and from the weak form of \eqref{eq:w} it immediately follows that $\|w\|_{\mathbb{H}^s(\Omega)} \le C \|f\|_{\mathbb{H}^{-s}(\Omega)}$. Collecting all the estimates we obtain \eqref{eq:contdepdata}.
\end{proof}

In Section \ref{s:apriori_dbc} we will be concerned with discretization error estimates for \eqref{eq:fracLap}. For that purpose we need to establish higher regularity for the solution $u=w+v$ given more regular data $f$ and $g$. Due to the fact that the solution $w$ to \eqref{eq:w} is formally given by
\[
	w=\sum_{k=1}^\infty \lambda_k^{-s} f_k \varphi_k\quad\text{with}\quad f_k=\int_\Omega f\varphi_k,
\]
we obtain that $w$ belongs to $\mathbb{H}^{1+s}(\Omega)$ for any $f\in \mathbb{H}^{1-s}(\Omega)$.
The results about higher regularity for the solution $v$ to \eqref{eq:v} are collected in the following lemma.
\begin{lemma}[Regularity of $v$]\label{l:regv}
	Let one of the following conditions be fulfilled:
	\begin{enumerate}[label=\emph{(\alph*)}]
		\item\label{it:1} $0 \leq s \leq \frac12$: $\Omega$ is Lipschitz, $g \in H^{s+\frac12}(\partial\Omega)$,
		\item\label{it:2} $\frac12 < s \leq 1$: $\Omega$ is quasi-convex, $g \in [\gamma_D(H^2(\Omega)), H^1(\partial\Omega)]_{2(1-s)}$,  
	\end{enumerate}
	where $\gamma_D$ denotes the Dirichlet trace operator. Then $v$ belongs to $H^{1+s}(\Omega)$ and fulfills
		\[
			\|v\|_{H^{1+s}(\Omega)}\leq C \|g\|_{\mathbb{D}^{s+\frac12}(\partial\Omega)}
		\]
	with a constant $C$ independent of $g$, and the trace space $\mathbb{D}^{s+\frac12}(\partial\Omega)$ defined by 
		\[
			\mathbb{D}^{s+\frac12}(\partial\Omega):=\begin{cases}
					H^{s+\frac12}(\partial\Omega) & \text{if }0\leq s\leq \frac12,\\
					[\gamma_D(H^2(\Omega)), H^1(\partial\Omega)]_{2(1-s)} & \text{if } \frac12 <s\leq 1.
				\end{cases}
		\]
\end{lemma}

	\begin{remark}\label{re:trace}
		Notice that by definition every quasi-convex domain is Lipschitz, therefore condition \ref{it:1} in Lemma~\ref{l:regv} also holds in quasi-convex domains. Moreover, when $\Omega$ is $C^{1,r}$ with $1/2 < r < 1$ (cf. \cite[Lemma 10.1]{MR2788354}), then $\gamma_D(H^2(\Omega)) = H^{3/2}(\partial\Omega)$, whence, the interpolation space in part \ref{it:2} of Lemma~\ref{l:regv} is
		\[
		[\gamma_D(H^2(\Omega)), H^1(\partial\Omega)]_{2(1-s)} 
		= H^{s+1/2}(\partial\Omega) . 
		\]
		Notice as well that $g\in \gamma_D(H^2(\Omega))$ implies that on each side/face $\Gamma_i$ of a polygonal/polyhedral domain $\Omega$ we have $g\in H^\frac32(\Gamma_i)$. Consequently, in case of polygonal/polyhedral domains, we conclude by real interpolation
			\[
				\|g\|_{H^{s+\frac12}(\Gamma_i)}\leq c \|g\|_{\mathbb{D}^{s+\frac12}(\partial\Omega)}
			\]
		for any $g\in \mathbb{D}^{s+\frac12}(\partial\Omega)$.
	\end{remark}

\begin{proof}
	When $0 \leq s \leq \frac12$, this result follows from  \cite[Theorem 5.3]{MR2788354}. Finally, when $\frac12 < s \leq 1$ we recall from \cite[Eq.~(10.16)-(10.17) in Theorem 10.4]{MR2788354}
	\begin{align*}
	g \in H^1(\partial\Omega) \quad           & \mbox{implies } v \in H^{3/2}(\Omega), \\
	g \in \gamma_D(H^2(\Omega)) \quad & \mbox{implies } v \in H^{2}(\Omega).
	\end{align*}
	Moreover, corresponding natural a priori estimates are valid. Using real interpolation we arrive at 
	\[
	g \in [\gamma_D(H^2(\Omega)), H^1(\partial\Omega)]_{2(1-s)} \quad \mbox{implies } v \in H^{1+s}(\Omega),
	\]
	which completes the proof. 
\end{proof}

%%%%%%%%%%%%%%%%%%%%%%%%%%%%%%%%%%%%%%%%%%%%%%%%%%%%%%%%%%%%%%%%%%%%%%%%%%%%%%%%%%%%%%%%%%%%%%%%%%%%%%%%%%%%%%%%%%%%%%%%%%
\subsection{The extended problem}\label{s:CSext}
%%%%%%%%%%%%%%%%%%%%%%%%%%%%%%%%%%%%%%%%%%%%%%%%%%%%%%%%%%%%%%%%%%%%%%%%%%%%%%%%%%%%%%%%%%%%%%%%%%%%%%%%%%%%%%%%%%%%%%%%%%

It is well-known that \eqref{eq:w} can equivalently be posed on a semi-infinite cylinder. This approach in $\mathbb{R}^n$ is due to Caffarelli and Silvestre \cite{LCaffarelli_LSilvestre_2007a}. The restriction to bounded domains was considered by Stinga-Torrea in \cite{PRStinga_JLTorrea_2010a}, see also \cite{XCabre_JTan_2010a, ACapella_JDavila_LDupaigne_YSire_2011a}. For the existence and uniqueness of a solution to the problem on the semi-infinite cylinder it is sufficient to consider $\Omega$ to be a bounded open set with Lipschitz boundary \cite[Theorem~2.5]{MR3489634}. 

We first introduce the required notation, we will follow \cite[section~3]{HAntil_JPfefferer_MWarma_2016a}. We denote by $\mathcal{C}$ the aforementioned semi-infinite cylinder with base $\Omega$, i.e., $\mathcal{C} = \Omega \times (0,\infty)$, and its lateral boundary $\partial_L\mathcal{C} := \partial\Omega \times [0,\infty)$. We also need to define a truncated cylinder: for $\mathcal{Y} > 0$, the truncated cylinder is given by $\mathcal{C}_\mathcal{Y}$. Additionally, we set $\partial_L\mathcal{C}_\mathcal{Y} := \partial\Omega \times [0,\mathcal{Y}]$. As $\mathcal{C}$ and $\mathcal{C}_\mathcal{Y}$ are objects in $\mathbb{R}^{n+1}$, we use $y$ to denote the extended variable, such that a vector $x' \in \mathbb{R}^n$ admits the representation $x' = (x_1, \dots, x_n, x_{n+1}) = (x,x_{n+1}) = (x,y)$ with $x_i \in \mathbb{R}$ for $i = 1, \dots, n+1$, $x \in \mathbb{R}^n$ and $y \in \mathbb{R}$.

Next we introduce the weighted Sobolev spaces with a degenerate/singular weight function $y^\alpha$, $\alpha \in (-1,1)$, see \cite[Section~2.1]{Turesson}, \cite{KO84}, and \cite[Theorem~1]{GU} for further discussion on such spaces. Towards this end, let $\mathcal{D}\subset \mathbb{R}^{n}\times[0,\infty)$ be an open set, such as $\mathcal{C}$ or $\mathcal{C}_\mathpzc{Y}$, then we define the weighted space $L^2(y^\alpha,\mathcal{D})$ as the space of all measurable functions defined on $\mathcal{D}$ with finite norm $\|w\|_{L^2(y^\alpha,\mathcal{D})}:=\|y^{\alpha/2}w\|_{L^2(\mathcal{D})}$. Similarly, using a standard multi-index notation, the space $H^1(y^\alpha,\mathcal{D})$ denotes the space of all measurable functions $w$ on $\mathcal{D}$ whose weak derivatives $D^\delta w$ exist for $|\delta|=1$ and fulfill
\[
	\|w\|_{H^1(y^\alpha,\mathcal{D})}:=\left(\sum_{|\delta|\leq 1}\|D^\delta w\|^2_{L^2(y^\alpha,\mathcal{D})}\right)^{1/2}<\infty.
\]
To study the extended problems we also need to introduce the space
\[
  \mathring{H}^1_L(y^\alpha,\mathcal{C}):=\{w\in H^1(y^\alpha,\mathcal{C}):w=0\text{ on } \partial_L\mathcal{C}\}.
\]
The space $\mathring{H}^1_L(y^\alpha,\mathcal{C}_\mathpzc{Y})$ is defined analogously, i.e.,
\[
\mathring{H}^1_L(y^\alpha,\mathcal{C}_\mathpzc{Y}):=\{w\in H^1(y^\alpha,\mathcal{C}_\mathpzc{Y}):w=0\text{ on } \partial_L\mathcal{C}_\mathpzc{Y}\cup \Omega\times \{\mathpzc{Y}\}\}.
\]
We finally state the extended problem in the weak form: Given $f\in \mathbb{H}^{-s}(\Omega)$, find $\mathcal{W}\in \mathring{H}^1_L(y^\alpha,\mathcal{C})$ such that 
\begin{equation}\label{eq:extendedweak}
\int_\mathcal{C} y^\alpha \nabla \mathcal{W}\cdot\nabla \Phi=d_s\langle f,\Phi	
\rangle_{\mathbb{H}^{-s}(\Omega),\mathbb{H}^s(\Omega)}\quad \forall 
\Phi\in\mathring{H}^1_L(y^\alpha,\mathcal{C})
\end{equation}
with $\alpha=1-2s$ and $d_s=2^\alpha \frac{\Gamma(1-s)}{\Gamma(s)}$, where we recall that $0<s<1$. That is, the function $\mathcal W\in \mathring{H}^1_L(y^\alpha,\mathcal{C})$ is a weak solution of the following  problem
\begin{equation}\label{edp}
\begin{cases}
\mbox{div}(y^\alpha\nabla \mathcal W)=0\;\;&\mbox{ in}\;\mathcal C,\\
\frac{\partial\mathcal W}{\partial\nu^\alpha}=d_s f\;\;\;&\mbox{ on }\;\Omega\times\{0\},
\end{cases}
\end{equation}
where we have set
\begin{align*}
\frac{\partial\mathcal W}{\partial\nu^\alpha}(x,0)=\lim_{y\to 0}y^\alpha\mathcal W_y(x,y)=\lim_{y\to 0}y^\alpha\frac{\partial\mathcal W(x,y)}{\partial y}.
\end{align*}
Even though the extended problem \eqref{eq:extendedweak} is local (in contrast to the nonlocal problem \eqref{eq:w}), however, a direct discretization is still challenging due to the semi-infinite computational domain $\mathcal{C}$. To overcome this, we employ the exponential decay of the solution $\mathcal{W}$ in certain norms as $y$ tends to infinity, see \cite{RHNochetto_EOtarola_AJSalgado_2014a}. This suggests truncating the semi-infinite cylinder, leading to a problem posed on the truncated cylinder $\mathcal{C}_\mathpzc{Y}$: Given $f \in \mathbb{H}^{-s}(\Omega)$, find $\mathcal{W}_{\mathpzc{Y}}  \in  \mathring{H}^1_L(y^\alpha,\mathcal{C}_\mathpzc{Y})$ such that 
\begin{equation}\label{eq:truncatedweak}
\int_{\mathcal{C}_\mathpzc{Y}} y^\alpha \nabla 
\mathcal{W}_\mathpzc{Y}\cdot\nabla \Phi=d_s\langle f,\Phi	
\rangle_{\mathbb{H}^{-s}(\Omega),\mathbb{H}^{s}(\Omega)}\quad \forall 
\Phi\in\mathring{H}^1_L(y^\alpha,\mathcal{C}_\mathpzc{Y}).
\end{equation}
We refer to \cite[Theorem~3.5]{RHNochetto_EOtarola_AJSalgado_2014a} for the estimate of the truncation error.

%%%%%%%%%%%%%%%%%%%%%%%%%%%%%%%%%%%%%%%%%%%%%%%%%%%%%%%%%%%%%%%%%%%%%%%%%%%%%%%%%%%%%%%%%%%%%%%%%%%%%%%%%%%%%%%%%%%%%%%%%%
\subsection{A Priori Error Estimates}\label{s:apriori_dbc}
%%%%%%%%%%%%%%%%%%%%%%%%%%%%%%%%%%%%%%%%%%%%%%%%%%%%%%%%%%%%%%%%%%%%%%%%%%%%%%%%%%%%%%%%%%%%%%%%%%%%%%%%%%%%%%%%%%%%%%%%%%

To get an approximation of $\mathcal{W}$, we apply the approach from \cite{RHNochetto_EOtarola_AJSalgado_2014a}, i.e. the truncated problem is discretized by a finite element method, and in order to obtain an approximation of $v$, we will use the approach described in \cite{MR2084239,NUM:NUM22057,apelnicaisepfefferer:2015-2extended} or equivalently a standard finite element method if the boundary datum is smooth enough. From here on, we assume that the domain $\Omega$ is convex and polygonal/polyhedral. 	

Due to the singular behavior of $\mathcal{W}$ towards the boundary $\Omega$, we will use anistropically refined meshes. We define these meshes as follows:
Let $\mathscr{T}_\Omega=\{K\}$ be a conforming and quasi-uniform triangulation of $\Omega$, where $K\in \mathbb{R}^n$ is an element that is isoparametrically equivalent either to the unit cube or to the unit simplex in $\mathbb{R}^n$. We assume $\# \mathscr{T}_\Omega \sim M^n$. Thus, the element size $h_{\mathscr{T}_\Omega}$ fulfills $h_{\mathscr{T}_\Omega}\sim M^{-1}$. The collection of all these meshes is denoted by $\mathbb{T}_\Omega$. Furthermore, let $\mathcal{I}_\mathpzc{Y}=\{I\}$ be a graded mesh of the interval $[0,\mathpzc{Y}]$ in the sense that $[0,\mathpzc{Y}]=\bigcup_{k=0}^{M-1}[y_k,y_{k+1}]$ with
\[
y_k=\left(\frac{k}{M}\right)^\gamma\mathpzc{Y},\quad k=0,\ldots,M,\quad \gamma>\frac{3}{1-\alpha}=\frac{3}{2s}>1.
\]
Now, the triangulations $\mathscr{T}_\mathpzc{Y}$ of the cylinder $\mathcal{C}_\mathpzc{Y}$ are constructed as tensor product triangulations by means of $\mathscr{T}_\Omega$ and $\mathcal{I}_\mathpzc{Y}$. The definitions of both imply $\# \mathscr{T}_\mathpzc{Y} \sim M^{n+1}$.
Finally, the collection of all those anisotropic meshes $\mathscr{T}_\mathpzc{Y}$ is denoted by 
$\mathbb{T}$.

Now, we define the finite element spaces posed on the previously introduced 
meshes. 
For every $\mathscr{T}_\mathpzc{Y}\in \mathbb{T}$ the finite element spaces 
$\mathbb{W}(\mathscr{T}_\mathpzc{Y})$ are now defined by
\[
\mathbb{W}(\mathscr{T}_\mathpzc{Y}):=\{\Phi\in C^0(\overline{ \mathcal{C}_\mathpzc{Y}}):\Phi|_{T}\in\mathcal{P}_1(K)\oplus\mathbb{P}_1(I)\ \forall \;T=K\times I\in \mathscr{T}_\mathpzc{Y},\ \Phi|_{\partial_L\mathcal{C}_\mathpzc{Y}}=0\}.
\]
In case that $K$ in the previous definition is a simplex then $\mathcal{P}_1(K)=\mathbb{P}_1(K)$, the set of polynomials of degree at most $1$. If $K$ is a cube then $\mathcal{P}_1(K)$ equals $\mathbb{Q}_1(K)$, the set of polynomials of degree at most 1 in 
each variable.

Using the just introduced notation, the finite element discretization of \eqref{eq:truncatedweak} is given by the function $\mathcal{W}_{\mathscr{T}_\mathpzc{Y}}\in\mathbb{W}(\mathscr{T}_\mathpzc{Y})$ which solves the variational identity
\begin{equation}\label{eq:truncateddiscrete}
\int_{\mathcal{C}_\mathpzc{Y}} y^\alpha \nabla 
\mathcal{W}_{\mathscr{T}_\mathpzc{Y}}\cdot\nabla \Phi =d_s\langle f,\Phi	
\rangle_{\mathbb{H}^{-s}(\Omega),\mathbb{H}^s(\Omega)}\quad \forall 
\Phi\in\mathbb{W}(\mathscr{T}_\mathpzc{Y})
\end{equation}
with $\alpha=1-2s$ and $d_s=2^\alpha \frac{\Gamma(1-s)}{\Gamma(s)}$, where we recall that $0<s<1$.

Next we are concerned with the discretization of \eqref{eq:v_ver}. Since we will assume that the boundary datum $g$ belongs at least to $H^{\frac12}(\partial \Omega)$, a standard finite element discretization is applicable, see e.g. \cite{Bartels2004}. More precisely, let 
	\[
	\mathbb{V} := \{ \phi \in C^0(\overline\Omega) : \phi|_{K} \in \mathcal{P}_1(K) \} , \quad
	\mathbb{V}_0 := \mathbb{V} \cap H^1_0(\Omega), \quad \mathbb{V}^\partial = \mathbb{V}|_{\partial\Omega}. 
	\]
Moreover, let $\Pi_{\mathscr{T}_\Omega}$ denote the $L^2$-projection into $\mathbb{V}$. Then we seek a discrete solution $v_{\mathscr{T}_\Omega}\in \mathbb{V}_* := \{v \in \mathbb{V} : v|_{\partial\Omega} = \Pi_{\mathscr{T}_\Omega} g\}$ which fulfills 
\begin{equation}\label{eq:vwdisc}
\int_\Omega \nabla v_{\mathscr{T}_\Omega} \cdot \nabla \phi= 0 \quad \forall 
\phi \in \mathbb{V}_0 .  
\end{equation}
	
Notice that in case that $g\notin H^{\frac12}(\partial \Omega)$, the weak formulation of \eqref{eq:v_ver} is not well-posed. However, the discretization \eqref{eq:vwdisc} is still reasonable and corresponding error estimates hold, see \cite{MR2084239,NUM:NUM22057,apelnicaisepfefferer:2015-2extended}.

Finally, we define the discrete solution to \eqref{eq:fracLap} as 
\begin{equation}\label{eq:discu}
 u_{\mathscr{T}_\Omega} = \mathcal{W}_{\mathscr{T}_\mathpzc{Y}}(\cdot,0) + v_{\mathscr{T}_\Omega},
\end{equation}
where $\mathcal{W}_{\mathscr{T}_\mathpzc{Y}}$ and $v_{\mathscr{T}_\Omega}$ solve \eqref{eq:truncateddiscrete} and \eqref{eq:vwdisc}, respectively. 

We conclude this section with the following theorem about discretization error estimates for $u_{\mathscr{T}_\Omega}$.

\begin{theorem}\label{t:fem_est}
Let $\Omega$ convex polygonal/polyhedral, $g \in \mathbb{D}^{s+\frac12}(\partial\Omega)$ satisfies the conditions of Lemma~\ref{l:regv}, and $f \in \mathbb{H}^{1-s}(\Omega)$. Moreover, let $u$ be the solution of \eqref{eq:fracLap} and let $u_{\mathscr{T}_\Omega}$ be as in \eqref{eq:discu} then there is a constant $C>0$ independent of the data such that 
\begin{equation}\label{eq:Hsu}
 \|u - u_{\mathscr{T}_\Omega}\|_{H^s(\Omega)} \le C |\log(\# \mathscr{T}_\mathpzc{Y})|^s(\# \mathscr{T}_\mathpzc{Y})^{-\frac{1}{(n+1)}} \left( \|f\|_{\mathbb{H}^{1-s}(\Omega)} + \|g\|_{\mathbb{D}^{s+\frac12}(\partial\Omega)} \right)
\end{equation}
and 
\begin{equation}\label{eq:L2u}
 \|u - u_{\mathscr{T}_\Omega}\|_{L^2(\Omega)} \le C |\log(\# \mathscr{T}_\mathpzc{Y})|^{2s}(\# \mathscr{T}_\mathpzc{Y})^{-\frac{(1+s)}{(n+1)}} \left( \|f\|_{\mathbb{H}^{1-s}(\Omega)} + \|g\|_{\mathbb{D}^{s+\frac12}(\partial\Omega)} \right) 
\end{equation}
provided that $\mathpzc{Y}\sim \log(\# \mathscr{T}_\mathpzc{Y})$.
\end{theorem}
\begin{proof}
After applying the triangle inequality we arrive at 
\[
 \|u - u_{\mathscr{T}_\Omega}\|_{H^s(\Omega)} 
 \le \| w - \mathcal{W}_{\mathscr{T}_\mathpzc{Y}}(\cdot,0) \|_{H^s(\Omega)}
  +  \| v - v_{\mathscr{T}_\Omega} \|_{H^s(\Omega)} . 
\]
We treat each term on the right-hand-side separately. Using 
Proposition~\ref{l:HsBHs} 
we obtain
\[
 \| w - \mathcal{W}_{\mathscr{T}_\mathpzc{Y}}(\cdot,0) \|_{H^s(\Omega)}
 \le C \| w - \mathcal{W}_{\mathscr{T}_\mathpzc{Y}}(\cdot,0) \|_{\mathbb{H}^s(\Omega)} .
\] 
Subsequently invoking \cite[Theorem~5.4 and Remark~5.5]{RHNochetto_EOtarola_AJSalgado_2014a} we arrive at
\[
 \| w - \mathcal{W}_{\mathscr{T}_\mathpzc{Y}}(\cdot,0) \|_{\mathbb{H}^s(\Omega)}
 \le C |\log(\# \mathscr{T}_\mathpzc{Y})|^s(\# \mathscr{T}_\mathpzc{Y})^{-\frac{1}{(n+1)}} \|f\|_{\mathbb{H}^{1-s}(\Omega)} . 
\]
Condensing the last two estimates we obtain 
\[
 \| w - \mathcal{W}_{\mathscr{T}_\mathpzc{Y}}(\cdot,0) \|_{H^s(\Omega)}
 \le C |\log(\# \mathscr{T}_\mathpzc{Y})|^s(\# \mathscr{T}_\mathpzc{Y})^{-\frac{1}{(n+1)}} \|f\|_{\mathbb{H}^{1-s}(\Omega)}.
\]
We now turn to $\| v - v_{\mathscr{T}_\Omega} \|_{H^s(\Omega)}$. Using classical arguments (see e.g. \cite{Bartels2004}), in combination with the regularity results of Lemma \ref{l:regv} and Remark \ref{re:trace}, we infer the estimates 
\begin{align*}
\|v-v_{\mathscr{T}_\Omega}\|_{H^1(\Omega)} &\le C h_{\mathscr{T}_\Omega}^{s} \|g\|_{\mathbb{D}^{s+\frac12}(\partial\Omega)}, \\
\|v-v_{\mathscr{T}_\Omega}\|_{L^2(\Omega)} &\le C h_{\mathscr{T}_\Omega}^{1+s} \|g\|_{\mathbb{D}^{s+\frac12}(\partial\Omega)}.
\end{align*} 
Due to the fact that $H^s(\Omega)$ is the interpolation space between $H^1(\Omega)$ and $L^2(\Omega)$, see \eqref{eq:inf_interp_Hs}, we obtain 
\[
 \|v-v_{\mathscr{T}_\Omega}\|_{H^s(\Omega)} \le C h_{\mathscr{T}_\Omega} \|g\|_{\mathbb{D}^{s+\frac12}(\partial\Omega)} ,
\]
where the constant $C > 0$ is independent of $h_{\mathscr{T}_\Omega}$ and the data. Collecting the estimates for $v$, $w$, we obtain \eqref{eq:Hsu} after having observed that $h_{\mathscr{T}_\Omega}\sim (\# \mathscr{T}_\mathpzc{Y})^{-\frac{1}{(n+1)}}$.

Finally, \eqref{eq:L2u} is due the $L^2$-estimate of $\mathcal{W}_{\mathscr{T}_\mathpzc{Y}}(\cdot,0)$ \cite[Proposition 4.7]{RHNochetto_EOtarola_AJSalgado_2014b} and the aforementioned $L^2$-estimate of $v_{\mathscr{T}_\Omega}$. 
\end{proof}

%%%%%%%%%%%%%%%%%%%%%%%%%%%%%%%%%%%%%%%%%%%%%%%%%%%%%%%%%%%%%%%%%%%%%%%%%%%%%%%%%%%%%%%%%%%%%%%%%%%%%%%%%%%%%%%%%%%%%%%%%%
\section{Application III: Fractional Equation with Neumann Boundary Condition}\label{s:app_n}
%%%%%%%%%%%%%%%%%%%%%%%%%%%%%%%%%%%%%%%%%%%%%%%%%%%%%%%%%%%%%%%%%%%%%%%%%%%%%%%%%%%%%%%%%%%%%%%%%%%%%%%%%%%%%%%%%%%%%%%%%%

Given data $f \in H^{s}(\Omega)^*$, $g \in \mathbb{N}^{s-\frac{3}{2}}(\partial\Omega)$ with
\[
\mathbb{N}^{s-\frac{3}{2}}(\partial\Omega):=\begin{cases} N^{\frac32-s}(\partial\Omega)^*&\text{for }s\in[0,\frac12)\\
H^{-1}(\partial\Omega)&\text{for }s=\frac12\\
H^{s-\frac32}(\partial\Omega)&\text{for }s\in(\frac12,1]
\end{cases},
\]
we seek a function $u \in H_{\int}^s(\Omega)$ satisfying
\begin{align}\label{eq:fracLap_n}
\begin{aligned}
 (-\Delta_N)^s u &= f \quad \mbox{in } \Omega , \\
  \partial_\nu u &= g \quad \mbox{on } \partial\Omega. 
\end{aligned}
\end{align}
We assume that the data $f$ and $g$ additionally fulfill the compatibility condition
\begin{equation}\label{eq:comcond}
\int_\Omega f + \int_{\partial\Omega} g = 0. 
\end{equation} 
Now, we proceed as in Section~\ref{s:app}. Given $g \in \mathbb{N}^{s-\frac{3}{2}}(\partial\Omega)$,  
we construct $v \in \mathbb{N}^s(\Omega)$ solving 
\begin{align}\label{eq:v_n}
\begin{aligned}
(-\Delta_N)^s v &= |\Omega|^{-1}\int_{\Omega}f  \quad \mbox{in } \Omega \\
\partial_\nu  v &= g  \quad \mbox{on } \partial\Omega ,
\end{aligned}
\end{align}
and given $f \in H^s(\Omega)^*$, we seek $w \in H^s_{\int}(\Omega)$ fulfilling
\begin{align}\label{eq:w_n}
\begin{aligned}
 (-\Delta_N)^s w &= f +|\Omega|^{-1}\int_{\partial\Omega}g  \quad \mbox{in } \Omega     ,    \\
  \partial_\nu w &= 0  \quad \mbox{on } \partial\Omega.
\end{aligned}
\end{align}
Finally, we have $u = w+v$. 

We will show that \eqref{eq:v_n} is equivalent to solving the following standard Laplace problem with nonzero Neumann boundary conditions  
\begin{equation}\label{eq:v_ver_n}
 -\Delta v = |\Omega|^{-1}\int_{\Omega}f \quad \mbox{in } \Omega, \quad \partial_\nu v = g \quad \mbox{on } \partial\Omega, 
\end{equation}
in the very-weak form or in the classical weak form if the regularity of the boundary datum guarantees its well-posedness. 

We start with introducing the very weak form of \eqref{eq:v_ver_n}. Given $g\in N^{\frac{3}{2}}(\partial\Omega)^*$, we are seeking a function $v\in H^0_{\int}(\Omega)$ fulfilling 
\begin{equation}\label{eq:v_veryweak_n}
	\int_{\Omega}v(-\Delta)\varphi= \int_{\partial\Omega}g\varphi
	\quad
	\forall \varphi\in V,
\end{equation}
where $V = \{ \varphi \in H^1_{\int}(\Omega) \cap H^2(\Omega):\; \partial_\nu \varphi = 0 \mbox{ on } \partial\Omega \}$. 

\begin{lemma}
\label{l:v_exist_n}
    Let $\Omega$ be a bounded, quasi-convex domain.
	For any $f$ and $g\in N^{\frac{3}{2}}(\partial\Omega)^*$ fulfilling \eqref{eq:comcond}, there exists an unique very weak 
	solution $v\in H^0_{\int}(\Omega)$ of \eqref{eq:v_ver_n}. For more regular boundary 
	data $g\in \mathbb{N}^{s-\frac{3}{2}}(\partial\Omega)$ where $s\in [0,1]$, the solution 
	belongs to $H^s_{\int}(\Omega)$ and admits the a priori estimate
	\[
		|v|_{H^s(\Omega)}\leq c \|g\|_{\mathbb{N}^{s-\frac{3}{2}}(\partial\Omega)}.
	\]
	Moreover, if s=1 then the very weak solution is actually a weak solution.
\end{lemma}
\begin{remark}
	Notice that due to Remark~\ref{r:N32space}, when $\Omega$ is $C^{1,r}$ with $r > 1/2$ we have $N^{\frac{3}{2}}(\partial\Omega) = H^{\frac{3}{2}}(\partial\Omega)$ and thus $\mathbb{N}^{s-\frac{3}{2}}(\partial\Omega) = H^{s-\frac{3}{2}}(\partial\Omega)$. Moreover, as for the Dirichlet problem, by employing similar arguments, in combination with \cite{MR2788354}, the results of Lemma \ref{l:v_exist} can be extended to general Lipschitz domains at least for $s\in[\frac12,1]$.
\end{remark}
\begin{proof}
The proof is similar to the proof of Lemma \ref{l:v_exist}. We only elaborate on the main differences. The proof of existence and uniqueness of a solution $v$ in $L^2(\Omega)$ with $\int_\Omega v=0$ is again based on the Babu\v{s}ka-Lax-Milgram theorem using the isomorphism
	\[
	\Delta\varphi\in L^2(\Omega),\ \partial_\nu\varphi=0,\ \int_\Omega\varphi=0\quad 
	\Leftrightarrow \quad \varphi\in V,
	\]
see \cite[Theorem~10.8]{MR2788354}, and Lemma~\ref{l:NspaceD}. The higher regularity can be deduced by real interpolation from classical regularity results for the solution of the corresponding weak formulation, which is actually a very weak solution due to the integration-by-parts formula, and the regularity results in $H^{\frac12}(\Omega)$ from \cite[Theorem 5.4]{MR2788354}.
\end{proof}

The equivalence between \eqref{eq:v_n} and \eqref{eq:v_ver_n} now follows along the same lines as in Theorem~\ref{t:veqver}. We collect this result in the following theorem.
\begin{theorem}\label{t:veqver_n}
	Let the assumptions of Lemma~\ref{l:v_exist_n} hold. Then solving problem 
	\eqref{eq:v_n} is equivalent to solving 
	problem \eqref{eq:v_ver_n} in the very weak sense. As a consequence, the results of Lemma \ref{l:v_exist_n} are valid for the solution of the fractional problem \eqref{eq:v_n}.
\end{theorem}

Finally, we conclude this section with the well-posedness of \eqref{eq:fracLap_n}.

\begin{theorem}[Existence and uniqueness]\label{th:uN}
    Let $\Omega$ be a bounded, quasi-convex domain. 
    If $f \in H^{s}(\Omega)^*$, $g \in \mathbb{N}^{s-\frac{3}{2}}(\partial\Omega)$ fulfill the compatibility condition \eqref{eq:comcond} then the system \eqref{eq:fracLap_n} has a unique solution $u \in H^s_{\int}(\Omega)$. In addition
\begin{equation}\label{eq:contdepdata_n}
 |u|_{H^s(\Omega)} \le C \left( \|f\|_{H^{s}(\Omega)^*} + \|g\|_{\mathbb{N}^{s-\frac{3}{2}}(\partial\Omega)} \right) .
\end{equation}
\end{theorem}
\begin{proof}
The proof is similar to that of Theorem~\ref{t:exist} and is omitted for brevity. 
\end{proof}

%%%%%%%%%%%%%%%%%%%%%%%%%%%%%%%%%%%%%%%%%%%%%%%%%%%%%%%%%%%%%%%%%%%%%%%%%%%%%%%%%%%%%%%%%%%%%%%%%%%%%%%%%%%%%%%%%%%%%%%%%%
\section{Application IV: Boundary Control Problems}\label{s:nbc}
%%%%%%%%%%%%%%%%%%%%%%%%%%%%%%%%%%%%%%%%%%%%%%%%%%%%%%%%%%%%%%%%%%%%%%%%%%%%%%%%%%%%%%%%%%%%%%%%%%%%%%%%%%%%%%%%%%%%%%%%%%
\subsection{Dirichlet boundary control problem}
Given $u_d \in L^2(\Omega)$ and $\alpha > 0$, we consider the following problem: minimize 
\begin{equation}\label{eq:obj}
J(u,z) := \frac{1}{2}\left( \| u - u_d \|_{L^2(\Omega)}^2 + \alpha \| q \|_{L^2(\partial\Omega)}^2 \right)
\end{equation}
subject to the state equation
\begin{align}\label{eq:state}
\begin{aligned}
(-\Delta_D)^s u &= 0 \quad \mbox{in } \Omega,  \\
u &= q \quad \mbox{on } \partial\Omega , 
\end{aligned}
\end{align}
and for given $a, b \in L^2(\partial\Omega)$ with $a(x) < b(x)$ for a.a. $x \in \partial\Omega$, the control $q$ belongs to the admissible set $Q_{ad}$ defined as 
\begin{equation}\label{eq:Qad}
Q_{ad} := \{ q \in L^2(\partial\Omega) : a(x) \le q(x) \le b(x) \ \mbox{for a.a. } x \in \partial\Omega \} . 
\end{equation} 
Notice that $L^2(\partial\Omega)\subset N^{\frac12}(\partial\Omega)^*$. Consequently,
owing to Theorem~\ref{t:veqver} we notice that 
the state equation \eqref{eq:state} is equivalent to 
\begin{align}\label{eq:state_vw}
\begin{aligned}
-\Delta u &= 0 \quad \mbox{in } \Omega,  \\
u &= q \quad \mbox{on } \Omega ,
\end{aligned}
\end{align}
where the latter is understood in the very-weak sense. 

Without going into further details we refer to \cite{CasasRaymond2006,MR3070527,AMPR2016} where the (numerical) analysis for this problem is carried out. The advantage of our characterization of fractional Laplacian is clear, i.e., it allows to equivalently rewrite the fractional optimal control problem into an optimal control problem which has been well studied. 

\subsection{Neumann boundary control problem}
Given $u_d \in L^2(\Omega)$ and $\alpha > 0$, we consider the following problem: minimize 
$J(u,q)$ as defined in \eqref{eq:obj} subject to the state equation
\begin{align}\label{eq:state_n}
\begin{aligned}
 (-\Delta_N)^s u &= 0 \quad \mbox{in } \Omega , \\
  \partial_\nu u &= q \quad \mbox{on } \partial\Omega . 
\end{aligned}
\end{align}
and the control $q \in Q_{ad}$ with $\int_{\partial\Omega}q=0$, where $Q_{ad}$ is defined in \eqref{eq:Qad}. Since $q \in L^2(\partial\Omega) \subset \mathbb{N}^{s-\frac{3}{2}}(\partial\Omega)$, the state equation \eqref{eq:state_n} is well-posed according to Theorem \ref{th:uN}. Moreover, it
is equivalent to 
\begin{align}\label{eq:state_vw_n}
\begin{aligned}
  -\Delta u &= 0 \quad \mbox{in } \Omega,  \\
  \partial_\nu u &=q \quad \mbox{on } \partial\Omega ,
\end{aligned}
\end{align}
where the latter can be understood in the classical weak sense. 

The optimization problem with constraints 
\begin{align}\label{eq:state_vw_n_2}
 \begin{aligned}
  -\Delta u + c u &= 0 \quad \mbox{in } \Omega,  \\
  \partial_\nu u &= q \quad \mbox{on } \partial\Omega ,
\end{aligned}
\end{align}
where $c >0 $, has been well studied, see \cite{MR2150243,MR2396868,MR2795722,HinzeMatthes:2008,MR3266952}.

%%%%%%%%%%%%%%%%%%%%%%%%%%%%%%%%%%%%%%%%%%%%%%%%%%%%%%%%%%%%%%%%%%%%%%%%%%%%%%%%%%%%%%%%%%%%%%%%%%%%%%%%%%%%%%%%%%%%%%%%%%
\section{Numerics}\label{s:numerics}
%%%%%%%%%%%%%%%%%%%%%%%%%%%%%%%%%%%%%%%%%%%%%%%%%%%%%%%%%%%%%%%%%%%%%%%%%%%%%%%%%%%%%%%%%%%%%%%%%%%%%%%%%%%%%%%%%%%%%%%%%%

Let $n = 2$. We verify the results of Theorem~\ref{t:fem_est} by two numerical examples. In the first example, we let the exact solution $w$ and $v$ to \eqref{eq:w} and \eqref{eq:v} to be smooth. In the second example we will take $v$ to be a nonsmooth function. All the computations were carried out in MATLAB under the $i$FEM library \cite{LChen_2009a}.

%%%%%%%%%%%%%%%%%%%%%%%%%%%%%%%%%%%%%%%%%%%%%%%%%%%%%%%%%%%%%%%%%%%%%%%%%%%%%%%%%%%%%%%%%%%%%%%%%%%%%%%%%%%%%%%%%%%%%%%%%%
\subsection{Example 1: Smooth Data}\label{s:ex1}
%%%%%%%%%%%%%%%%%%%%%%%%%%%%%%%%%%%%%%%%%%%%%%%%%%%%%%%%%%%%%%%%%%%%%%%%%%%%%%%%%%%%%%%%%%%%%%%%%%%%%%%%%%%%%%%%%%%%%%%%%%

Let $\Omega = (0,1)^2$. Under this setting, the eigenvalues and eigenfunctions of $-\Delta_{D,0}$ are:
\[
 \lambda_{k,l} = \pi^2 (k^2 + l^2), \quad 
 \varphi_{k,l} = \sin(k\pi x_1) \sin(l\pi x_2)  . 
\]
Setting $f = \sin(2\pi x_1) \sin(2\pi x_2)$ then the exact solution of \eqref{eq:w} is 
\[
 w = \lambda_{2,2}^{-s} \sin(2\pi x_1) \sin(2\pi x_2).
\] 
We let $v = x_1 + x_2$ and $g = x_1 + x_2$. Recall that $u = v+w$. As $g$ is smooth, the approximation error will be dominated by the error in $w$. 

Recall that $\|u-u_{\mathscr{T}_\Omega}\|_{H^s(\Omega)} \le \|v-v_{\mathscr{T}_\Omega}\|_{H^s(\Omega)} + \|w-\mathcal W_{\mathscr{T}_\mathpzc{Y}}\|_{H^s(\Omega)}$, where $u$, $v$, and $w$ are the exact solutions and $u_{\mathscr{T}_\Omega}$, $v_{\mathscr{T}_\Omega}$, and $\mathcal W_{\mathscr{T}_\mathpzc{Y}}$ are the approximated solutions. Recall from Proposition~\ref{l:HsBHs} that $\|w-\mathcal W_{\mathscr{T}_\mathpzc{Y}}\|_{H^s(\Omega)} \le C \|w-\mathcal W_{\mathscr{T}_\mathpzc{Y}}\|_{\mathbb{H}^s(\Omega)}$. Then using the extension, in conjunction with Galerkin-orthogonality, it is straightforward to approximate the $\mathbb{H}^s(\Omega)$-norm
\[
 \|w-\mathcal W_{\mathscr{T}_\mathpzc{Y}}\|_{\mathbb{H}^s(\Omega)}^2 \le C \|\nabla(\mathcal W-\mathcal W_{\mathscr{T}_\mathpzc{Y}})\|^2_{L^2(y^\alpha,\mathcal{C})} = d_s \int_\Omega f (w-W_{\mathscr{T}_\Omega}) \;dx .
\]
However, it is more delicate to compute $\|v-v_{\mathscr{T}_\Omega}\|_{H^s(\Omega)}$, for instance see \cite{MR2318290, MR2551152}. To accomplish this we first solve the generalized eigenvalue problem $\textbf{A}\textbf{x} = \lambda \textbf{M} \textbf{x}$, where $\textbf{A}$ and $\textbf{M}$ denotes the stiffness and mass matrices on $\Omega$. If $\textbf{v}$ and $\textbf{v}_{\mathscr{T}_\Omega}$ denotes the nodal values of the exact $v$ and approximated $v_{\mathscr{T}_\Omega}$ then we take 
\[\left(\|v-v_{\mathscr{T}_\Omega}\|^2_{L^2(\Omega)}+(\textbf{v}-\textbf{v}_{\mathscr{T}_\Omega})^T(\textbf{M}\textbf{V})^T\textbf{D}^s (\textbf{M}\textbf{V}) (\textbf{v}-\textbf{v}_{\mathscr{T}_\Omega}) \right)^{\frac12}
\]
as an approximation of $\|v-v_{\mathscr{T}_\Omega}\|_{H^s(\Omega)}$, where $\textbf{D}$ is the diagonal matrix with eigenvalues and the columns of matrix $\textbf{V}$ contains the eigenvectors of the aforementioned generalized eigenvalue problem.

Figure~\ref{f:ex1} (left) illustrates the $H^s$-norm, computed as described above. Figure~\ref{f:ex1} (right) shows the $L^2$-norm  of the error between the $u$ and $u_{\mathscr{T}_\Omega}$. 
As expected we observe $(\# \mathscr{T}_\mathpzc{Y})^{-\frac{1}{3}}$ rate in the former case. In the latter case, we observe a rate $(\# \mathscr{T}_\mathpzc{Y})^{-\frac{2}{3}}$ which is higher than the stated rate in Theorem~\ref{t:fem_est}. However, this is not a surprise as we already observed this in \cite{HAntil_JPfefferer_MWarma_2016a}, recall that our result for $L^2$-norm rely on \cite[Proposition~4.7]{RHNochetto_EOtarola_AJSalgado_2014b}.

\begin{figure}[h!]
\centering
\includegraphics[width=0.49\textwidth]{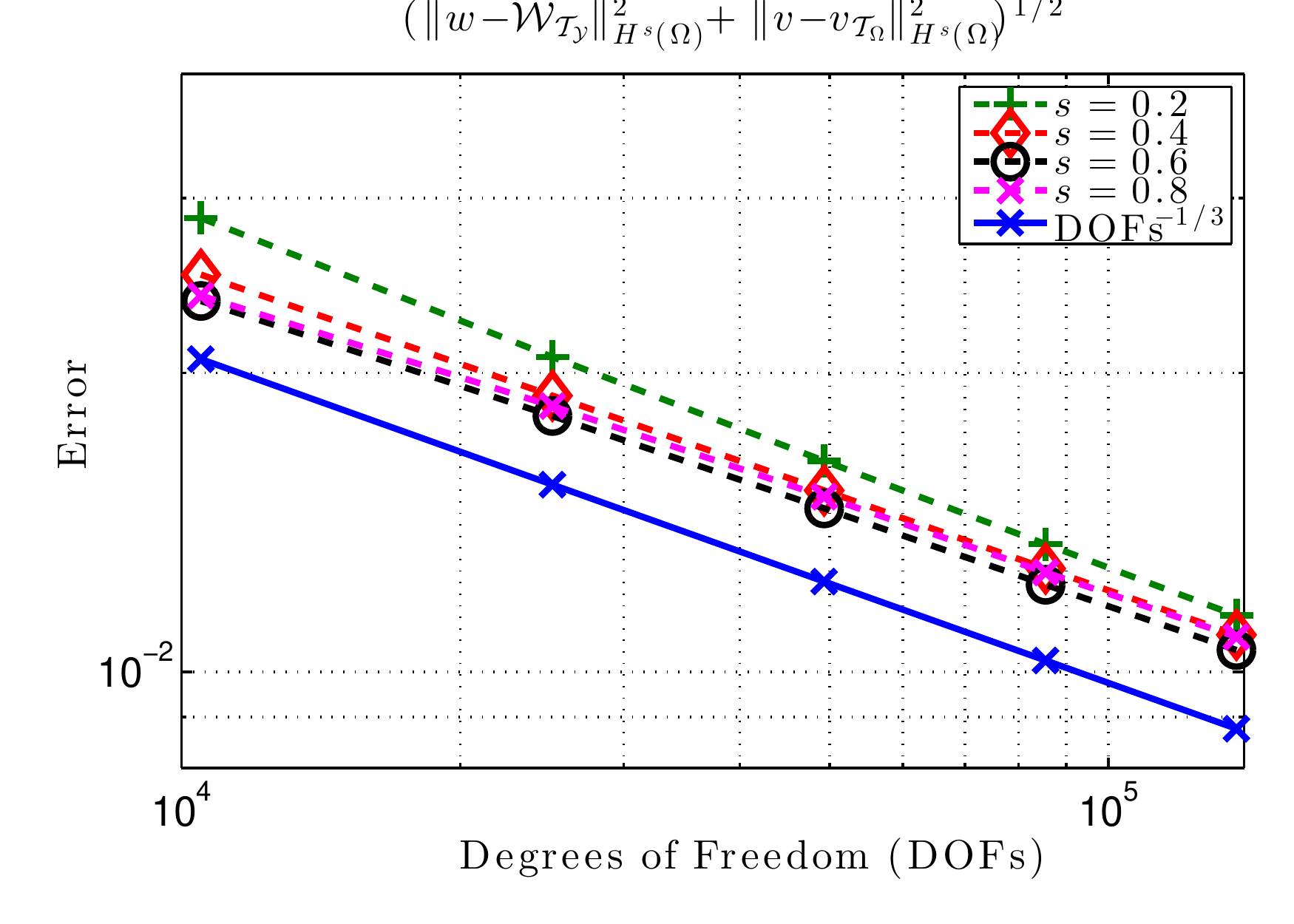}
\includegraphics[width=0.49\textwidth]{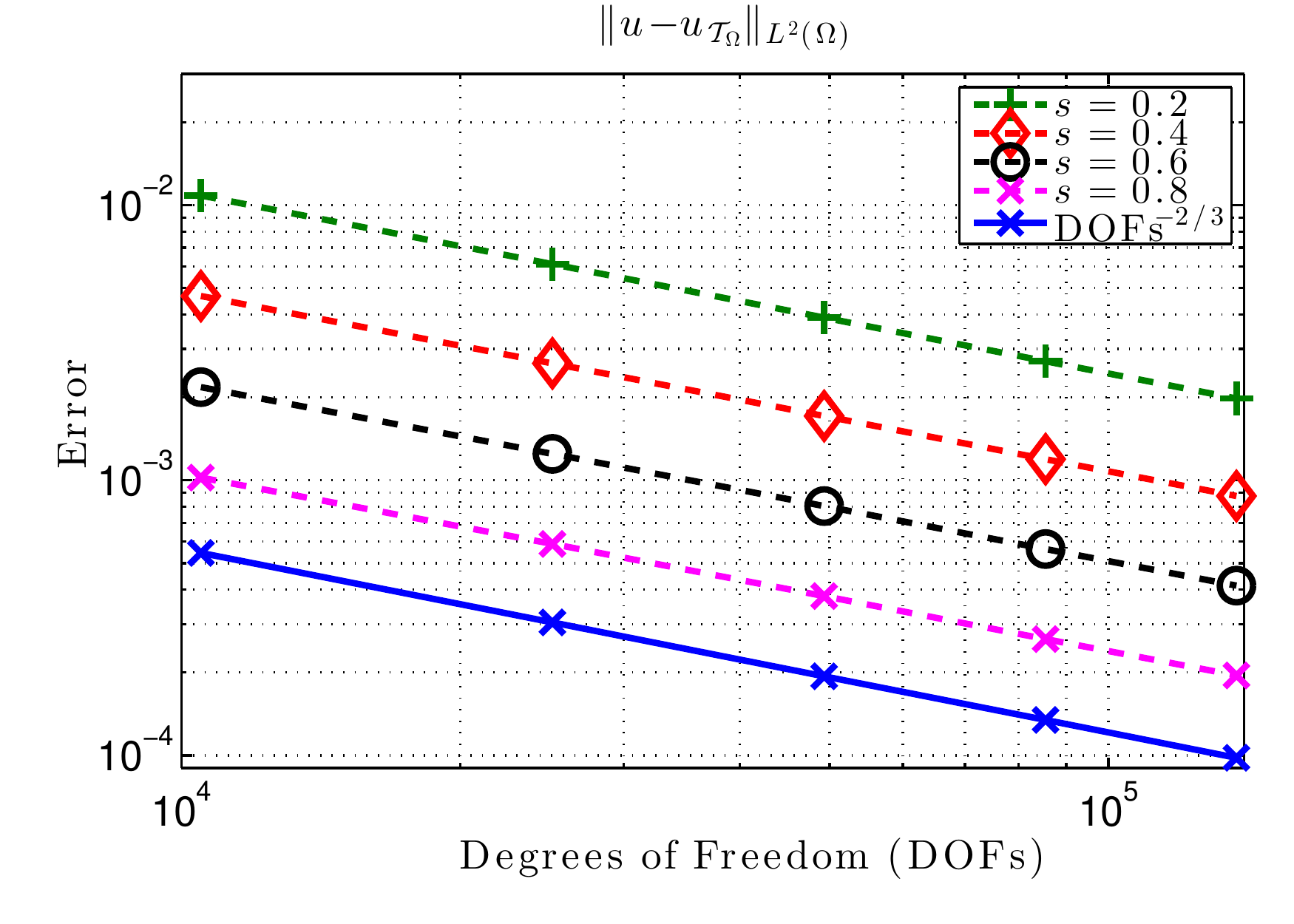}
\caption{\label{f:ex1}Rate of convergence on anisotropic meshes for $n=2$ and $s = 0.2,0.4,0.6,$ and $s=0.8$ is shown. The blue line is the reference line. The panel on the left shows $H^s$-error, in all cases we recover $(\# \mathscr{T}_\mathpzc{Y})^{-\frac{1}{3}}$. The right panel shows the $L^2$-error which decays as $(\# \mathscr{T}_\mathpzc{Y})^{-\frac{2}{3}}$.}
\end{figure}

%%%%%%%%%%%%%%%%%%%%%%%%%%%%%%%%%%%%%%%%%%%%%%%%%%%%%%%%%%%%%%%%%%%%%%%%%%%%%%%%%%%%%%%%%%%%%%%%%%%%%%%%%%%%%%%%%%%%%%%%%%
\subsection{Example 2: Nonsmooth Data}
%%%%%%%%%%%%%%%%%%%%%%%%%%%%%%%%%%%%%%%%%%%%%%%%%%%%%%%%%%%%%%%%%%%%%%%%%%%%%%%%%%%%%%%%%%%%%%%%%%%%%%%%%%%%%%%%%%%%%%%%%%

We let $\Omega = (0,1)^2$. Moreover, let $w$ and $f$ be the same as in Section~\ref{s:ex1} and we choose the boundary datum 
\[
g = r^{0.4999} \sin(0.4999 \ \theta) .
\]
This function belongs to $H^{1-\epsilon}(\partial\Omega)$ for every $\epsilon>0.0001$. The exact $v$ is simply
\[
v = r^{0.4999} \sin(0.4999 \ \theta) .
\]
Then $u = w+v$. In view of the regularity of $g$, we expect the approximation error of $u$ to be dominated by the approximation error in $v$ if $s>0.4999$. On the other hand, if $s<0.4999$, the approximation error of $w$ will dominate. More precisely, we expect in the former case a rate of about 
$(\# \mathscr{T}_\mathpzc{Y})^{-\frac{1}{3}\big(\frac{3}{2}-s\big)}$
in the $H^s(\Omega)$-norm. In the latter case, we expect a convergence rate of $(\# \mathscr{T}_\mathpzc{Y})^{-\frac{1}{3}}$ in the $H^s(\Omega)$-norm as in the foregoing example. 
Figure~\ref{f:ex2} confirms this.

\begin{figure}[h!]
\centering
\includegraphics[width=0.49\textwidth]{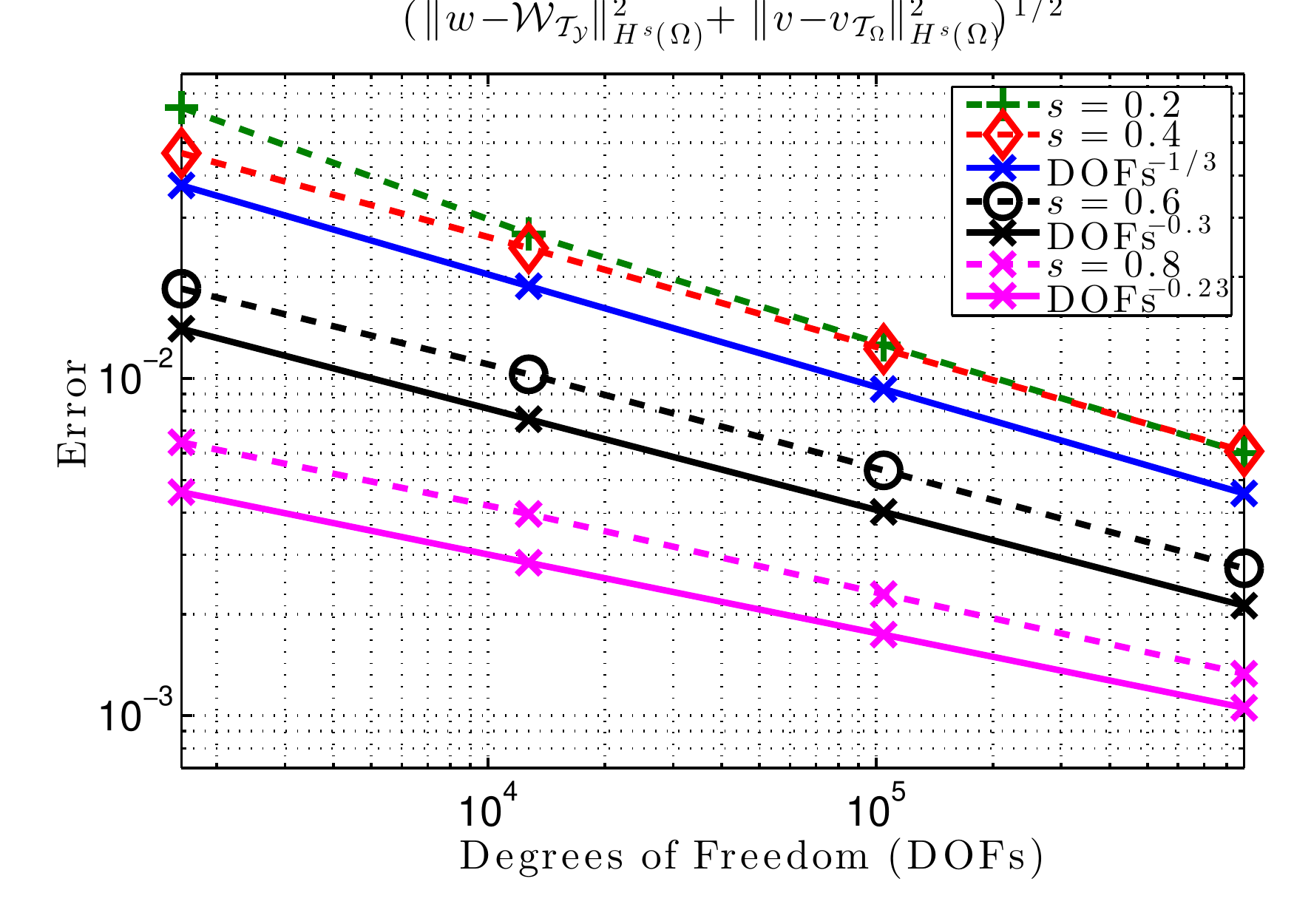}
\caption{\label{f:ex2} Rate of convergence on anisotropic meshes for $n=2$ and $s = 0.2,0.4,0.6,$ and $s=0.8$ is shown (dotted line). Starting from the top, the first solid line is the reference line with rate $(\# \mathscr{T}_\mathpzc{Y})^{-\frac{1}{3}}$. The second and third solid lines shows the rate $(\# \mathscr{T}_\mathpzc{Y})^{-\frac{1}{3}\big(\frac{3}{2}-s\big)}$ for $s = 0.6$ and $s = 0.8$, respectively.
}
\end{figure}

%%%%%%%%%%%%%%%%%%%%%%%%%%%%%%%%%%%%%%%%%%%%%%%%%%%%%%%%%%%%%%%%%%%%%%%%%%%%%%%%%%%%%%%%%%%%%%%%%%%%%%%%%%%%%%%%%%%%%%%%%%
\section{Further Extensions: General Second Order Elliptic Operators}\label{s:fext}
%%%%%%%%%%%%%%%%%%%%%%%%%%%%%%%%%%%%%%%%%%%%%%%%%%%%%%%%%%%%%%%%%%%%%%%%%%%%%%%%%%%%%%%%%%%%%%%%%%%%%%%%%%%%%%%%%%%%%%%%%%

We notice that our Definitions \ref{d:nzbc} and \ref{d:nzbc_n} immediately extend to general second order fractional operators. More precisely, let the general second order elliptic operator $\mathcal{L}$ be given as
\begin{equation}\label{eq:calL}
 \mathcal{L} u = -\mbox{div}(A \nabla u) \quad \mbox{in } \Omega.
\end{equation} 
Here, the coefficients $a_{ij}$ are measurable, belong to $L^\infty(\Omega)$, are symmetric, that is,
\begin{align*}
a_{ij}(x)=a_{ji}(x)\;\forall \; i,j=1,\ldots,n \mbox{ and for  a.e. }  x \in\Omega,
\end{align*}
and satisfy the ellipticity condition, that is, there exists a constant $\gamma>0$ such that
\begin{align*}
\sum_{i,j=1}^n a_{ij}(x)\xi_i\xi_j\ge \gamma|\xi|^2,\;\;\forall\;\xi\in \mathbb{R}^n.
\end{align*}
Moreover, we use $\partial_\nu^\mathcal{L}u$ to denote the conormal derivative of $u$, i.e.,
\begin{align}\label{co-nromal}
\partial_\nu^{\mathcal{L}}u=\sum_{j=1}^n\Big(\sum_{i=1}^n a_{ij}(x)D_iu\Big)\nu_j.
\end{align}
The fractional operators corresponding to $\mathcal{L}$ are defined as follows.

\begin{definition}[nonzero Dirichlet]
For $s\in(0,1)$, we define the spectral fractional Dirichlet Laplacian on $C^\infty(\bar\Omega)$ by 
\begin{equation}\label{eq:fLap_g_A}
\mathcal{L}_D^s u := \sum_{k=1}^\infty \left(\lambda_k^s \int_\Omega u \varphi_k 
     + \lambda_k^{s-1} \int_{\partial\Omega} u \partial_\nu^\mathcal{L} \varphi_k \right) \varphi_k ,
\end{equation}
where $(\lambda_k,\varphi_k)$ are the eigenvalue-eigenvector pairs of $\mathcal{L}$ with $\varphi_k|_{\partial\Omega} = 0$. 
\end{definition}

As we showed in Section~\ref{s:nzbc}, the operator 
$\mathcal{L}_D^s$ can be extended to an operator mapping from
\[
\mathbb{D}^s(\Omega):=\{u\in L^2(\Omega):\; 
\sum_{k=1}^\infty 
\lambda_k^{s}\left(u_{\Omega,k} + \lambda_k^{-1} u_{\partial\Omega,k} 
\right)^2<\infty\}
\]
to $\mathbb{H}^{-s}(\Omega)$, where 
$
	u_{\Omega,k}=\int_\Omega u \varphi_k$ and 	
$	u_{\partial\Omega,k}=\int_{\partial\Omega} u \partial_\nu^\mathcal{L} \varphi_k.
$

\begin{definition}[nonzero Neumann]
For $s\in(0,1)$, we define the spectral fractional Neumann Laplacian on $C^\infty(\bar{\Omega})$ 
by 
\begin{equation}\label{eq:fLap_g_An}
\mathcal{L}_N^s u := \sum_{k=2}^\infty \left(\mu_k^s \int_\Omega u \psi_k 
     - \mu_k^{s-1} \int_{\partial\Omega} \partial_\nu^\mathcal{L} u  \psi_k \right) \psi_k 
     -|\Omega|^{-1}\int_{\partial\Omega} \partial_\nu^\mathcal{L} u, 
\end{equation}
where $(\mu_k,\psi_k)$ are the eigenvalue-eigenvector pairs of $\mathcal{L}$ with $\partial_\nu^\mathcal{L} \psi_k = 0$.
\end{definition}

Again, as in Section~\ref{s:nn}, we set $u_{\Omega,k}=\int_\Omega u \psi_k$ and $u_{\partial\Omega,k}=\int_{\partial\Omega} \partial_\nu^\mathcal{L} u \psi_k$. Then, if we assume $\int_{\partial\Omega} \partial_\nu^\mathcal{L} u=0$, the operator 
$\mathcal{L}_N^s$ is extendable to an operator mapping from
\[
\mathbb{N}^s(\Omega):=\{u=\sum_{j=2}^\infty u_j\psi_j\in L^2(\Omega):\; 
\sum_{k=2}^\infty 
\mu_k^{s}\left(u_{\Omega,k} - \mu_k^{-1} u_{\partial\Omega,k} 
\right)^2<\infty\}
\]
to ${H}_{\int}^{-s}(\Omega)$. 

\begin{remark}
	For $c \in L^\infty(\Omega)$ and $c(x) > 0$ for a.a. $x\in\Omega$, we can further generalize $\mathcal{L}$ in \eqref{eq:calL} to $\mathcal{L} u = -\mbox{div}(A \nabla u) + c u$. The definitions above of fractional operators remain intact with the obvious modification in the Neumann case. 
\end{remark}

%%%%%%%%%%%%%%%%%%%%%%%%%%%%%%%%%%%%%%%%%%%%%%%%
\section*{Acknowledgement}
%%%%%%%%%%%%%%%%%%%%%%%%%%%%%%%%%%%%%%%%%%%%%%%%
We thank Boris Vexler, Pablo Stinga, and Mahamadi Warma for several fruitful discussions.

%%%%%%%%%%%%%%%%%%%%%%%%%%%%%%%%%%%%%%%%%%%%%%%%
\bibliographystyle{plain}
\bibliography{references}

\end{document}